\documentclass[final,a4paper,11pt]{article}
\usepackage{graphicx}
\usepackage{mathptmx}
\usepackage{amssymb,amsmath, amsfonts,amsthm}
\usepackage{a4wide}
\usepackage{color}

\newcommand{\email}[1]{{\tt #1}}
\newcommand{\R}{\mathbb{R}}

\newcommand{\norm}[1]{\|#1\|}
\newcommand{\bnorm}[1]{\big\|#1\big\|}

\newcommand{\dist}[1]{{\rm dist}(#1)}

\newcommand{\mv}{\,\mid\,}
\newcommand{\B}{{\cal B}}

\newcommand{\J}{{\cal J}}
\newcommand{\K}{{\cal K}}

\newcommand{\Sp}{{\cal S}}
\newcommand{\F}{{\cal F}}
\newcommand{\Lag}{{\cal L}}
\newcommand{\setto}[1]{\mathop{\rightarrow}\limits^#1}

\newcommand{\skalp}[1]{\langle #1\rangle}

\newcommand{\xb}{\bar x}
\newcommand{\yb}{\bar y}
\newcommand{\zb}{\bar z}

\newcommand{\pb}{\bar p}
\newcommand{\lb}{\bar\lambda}

\newcommand{\AT}[2]{{\textstyle{#1\atop#2}}}

\newcommand{\oo}{o}

\newcommand{\argmax}{\mathop{\rm arg\,max}\limits}

\newcommand{\lin}{{\rm lin\,}}
\newcommand{\Span}{{\rm sp\,}}

\newcommand{\ri}{{\rm ri\,}}
\newcommand{\inn}{{\rm int\,}}

\newcommand{\gph}{\mathrm{gph}\,}
\newcommand{\dom}{\mathrm{dom}\,}
\newcommand{\tto}{\rightrightarrows}
\newcommand{\Limsup}{\mathop{{\rm Lim}\,{\rm sup}}}
\newcommand{\myvec}[1]{\left(\begin{array}{c}#1\end{array}\right)}

\newtheorem{theorem}{Theorem}[section]
\newtheorem{proposition}[theorem]{Proposition}
\newtheorem{remark}[theorem]{Remark}
\newtheorem{lemma}[theorem]{Lemma}
\newtheorem{corollary}[theorem]{Corollary}
\newtheorem{definition}[theorem]{Definition}
\newtheorem{example}[theorem]{Example}
\newtheorem{assumption}{Assumption}

\title{Stability analysis for parameterized variational systems with implicit constraints}
\author{Mat\'{u}\v{s} Benko\thanks{
	      Institute of Computational Mathematics, Johannes Kepler University Linz,
              A-4040 Linz, Austria, \email{benko@numa.uni-linz.ac.at}}\and Helmut Gfrerer\thanks{Institute of Computational Mathematics, Johannes Kepler University Linz, A-4040 Linz, Austria; \email{helmut.gfrerer@jku.at}}
 \and   Ji\v{r}\'{i} V. Outrata\thanks{Institute of Information Theory and Automation, Academy of Sciences of the Czech Republic, 18208 Prague, Czech Republic, and Centre for
              Informatics and Applied Optimization, Federation University of Australia, POB 663, Ballarat,  Vic 3350, Australia,  \email{outrata@utia.cas.cz}}}
\date{}

\begin{document}
\maketitle
 {\bf Abstract.} In the paper we provide new conditions ensuring the isolated calmness property and the Aubin property
of parameterized variational systems with constraints depending, apart from the parameter, also on the solution
itself. Such systems include, e.g., quasi-variational inequalities and implicit complementarity problems. Concerning
the Aubin property, possible restrictions imposed on the parameter are also admitted. Throughout the paper, tools
from the directional limiting generalized differential calculus are employed enabling us to impose only rather weak (non-
restrictive) qualification conditions. Despite the very general problem setting, the resulting conditions are workable
as documented by some academic examples.

{\bf Key words.}
 parameterized variational system, solution map, Aubin property, isolated calmness property

{\bf AMS Subject classification.}
 49J53, 90C31, 90C46

\section{Introduction}

 In variational analysis, a great effort has been devoted to the study of stability and sensitivity of solution maps to parameter-dependent optimization and equilibrium problems. In particular, the researchers have investigated various Lipschitzian properties
of these maps around given reference points. To obtain useful results, one employs typically some efficient tools of generalized differentiation discussed in a detailed way in the monographs \cite{DoRo14,KlKum02, Mo06a, Mo18,RoWe98}. Starting from 2011, the available arsenal of these tools includes also the calculus of directional limiting normal cones and coderivatives which enables us in some cases a finer analysis of parametric equilibria than its non-directional counterpart. This new theory has been initiated in \cite{GiMo11} and then thoroughly developed in a number of papers authored and co-authored by H. Gfrerer \cite{BeGfrOut18,Gfr11,Gfr13a,Gfr13b,Gfr14a,GfrKl16,GfrMo17,GfrOut16,GfrOut18}.

In particular, in \cite{GfrOut16} one finds rather weak (non-restrictive) conditions ensuring the calmness and the Aubin property of general implicitly defined multifunctions. The criterion for the Aubin property has then been worked out in \cite{GfrOut17} for a class of parametric variational systems with fixed (non-perturbed) constraint sets and in \cite{GfrOut18} for systems with implicit parameter-dependent constraints. The model from \cite{GfrOut18} was investigated already in \cite{MoOut07} by using the (classical) generalized differential calculus of B. Mordukhovich. It encompasses quasi-variational inequalities (QVIs), implicit complementarity problems
and also standard variational inequalities of the first kind with parameter-dependent constraints.

In this paper we consider the same model as in \cite{MoOut07} and \cite{GfrOut18} but remove the (rather severe) non-degeneracy-type assumption imposed in \cite{GfrOut18} on the constraint system. Instead of it, we make use of a (much weaker) metric inequality stated in Assumption \ref{AssA1}. Further, we analyze now not just the standard Aubin property of the considered solution map, denoted by $S$, but the Aubin property {\em relative} to a given set of feasible parameters. Clearly, S may enjoy this type of Lipschitzian stability even when the standard Aubin property is violated. Finally, we provide in this paper also a new condition, ensuring the isolated calmness of $S$.

The structure of the considered constraint system has enabled us to employ some strong results from \cite{DoRo04, DoRo14} and \cite{GfrOut16} concerning tangents and normals to the graph of the normal-cone mapping associated with a convex polyhedral set. More precisely, these tangents and normals can be expressed via some faces of an associated critical cone. This representation substantially contributes to the workability of final conditions ensuring the Aubin property of $S$. In addition, also some other statements in connection with directional non-degeneracy and directional metric regularity could be formulated in terms of these faces.

The plan of the paper is as follows. Sections 2.1 and 2.2 provide the reader with basic notions of the standard and directional
generalized differential calculus and with some basic facts about those Lipschitzian stability properties which are extensively
used throughout the whole paper. Section 2.3 contains the necessary background from the theory of convex polyhedral sets and polyhedral multifunctions. The last preliminary Section 2.4 is then devoted to the directional metric subregularity of a
particular multifunction, which arises later as a qualification condition, and to the new notion of directional non-degeneracy of a
constraint system, playing a central role in the subsequent development. Section 3 concerns the general model of an
implicitly defined multifunction considered in \cite{GfrOut16}. In this framework we find there a directional variant of the Levy-Rockafellar
characterization of the isolated calmness property and a counterpart of \cite[Theorem 4.4]{GfrOut16} corresponding to the Aubin
property relative to a set of feasible parameters. In the rest of the paper these statements are worked out for the considered variational system with implicit constraints. So, in Section 4 the respective graphical derivative is computed, which is a basis for the formulation of the final condition ensuring the isolated calmness property of $S$
presented in Section 5. Therafter, in Section 6 one finds a new workable sufficient condition guaranteeing the Aubin property of S relative to a given set of feasible parameters. Both these final results as well as some other important statements are illustrated by examples.

There are well-known equilibria in economy and mechanics modeled by QVIs and implicit complementarity problems, cf.  \cite{BaiCa84}. As an example, let us mention the generalized Nash equilibrium problems (GNEPs) which describe, e.g., the behavior of agents acting on markets with a limited amount of resources. Very often, these equilibria depend on some uncertain data which can be viewed as parameters. The results of this paper can then be used in {\em post-optimal analysis} of such equilibria, where the stability issues are of ultimate importance.

Given a set-valued mapping $M:\R^l\times\R^n\tto\R^m$, the general implicitly defined multifunction analyzed in \cite{GfrOut16} is given by the relation
\begin{equation}\label{EqGeneralInclusion}0\in M(p,x).\end{equation}
We are going to analyze  the associated {\em solution mapping} $S:\R^l\tto\R^n$ defined by
\begin{equation}\label{EqGeneralSolMap} S(p):=\{x\in\R^n\mv 0\in M(p,x)\}.\end{equation}
The variational system investigated in \cite{MoOut07} and \cite{GfrOut18} attains the form
\begin{equation}
  \label{EqVarSystem}0\in M(p,x):=f(p,x)+\widehat N_{\Gamma(p,x)}(x)\ \text{with}\ \Gamma(p,x):=\{z\mv g(p,x,z)\in D\},
\end{equation}
where $f:\R^l\times\R^n\to\R^n$ is continuously differentiable, $g:\R^l\times\R^n\times\R^n\to\R^s$ is twice continuously differentiable and $D\subset\R^s$ is a convex polyhedral set.

The following notation is employed. Given a set $A \subset \R^{n}$, $\Span A$ stands for the linear hull of $A$, $\ri A$ is the relative interior of $A$ and $A^{\circ}$ is the (negative) polar of $A$. We denote by $\dist{x,A}:=\inf_{y\in A}\norm{x-y}$ the usual point to set distance with the convention $\dist{x,\emptyset}=\infty$.
For a sequence $x_k$, $x_k \setto{A} \xb$
stands for $x_k \to \xb$ with $x_k \in A$.
For a convex cone $K,~ \lin~ K$ denotes the {\em lineality space} of $K$, i.e., the set $K \cap(- K)$. Further, $\B_{\R^n}$, $\Sp_{\R^n}$  is the unit ball and the unit sphere in $\R^n$, respectively. Given a vector $a \in \R^{n}$, $[a]$ is the linear space generated by $a$ and $[a]^{\perp}$ stands for the orthogonal complement to $[a]$.
Finally, given a set-valued map $F:\R^n\tto\R^m$,
$\gph F:=\{(x,y)\in \R^n \times \R^m\mv y\in F(x)\}$
stands for the graph of $F$ and $\Limsup_{x\to \xb}F(x)$
denotes the outer set limit in the sense of Painlev\'{e}-Kuratowski.

\section{Preliminaries}
\subsection{Variational geometry and generalized differentiation}
We start by recalling several definitions and results from variational analysis.
Let $\Omega\subset\R^n$ be an arbitrary closed set and $\xb\in\Omega$.
The {\em contingent} (also called  {\em Bouligand} or {\em tangent}) {\em cone} to $\Omega$ at $\xb$,
denoted by $T_{\Omega}(\xb)$, is given by
\[T_{\Omega}(\xb):=\{u\in \R^n\mv \exists (u_k)\to u, (t_k)\downarrow 0: \xb+t_ku_k\in\Omega \ \forall k\}.\]
A tangent $u\in T_\Omega(\xb)$ is called {\em derivable} if $\dist{\xb+tu,\Omega}=\oo(t)$.

We denote by
\begin{equation*}
\widehat N_{\Omega}(\xb):=T_{\Omega}(\xb)^\circ
\end{equation*}
the {\em Fr\'echet} ({\em regular}) {\em normal cone} to $\Omega$ at $\xb$.
The {\em limiting} ({\em Mordukhovich}) {\em normal cone} to $\Omega$ at $\xb$ is
defined by
\[N_{\Omega}(\xb):=\{x^* \in\R^n \mv \exists (x_k)\setto{\Omega}\xb,\ (x_k^*)\to x^*: x_k^*\in \widehat N_{\Omega}(x_k) \ \forall k\}.\]
Finally, given a direction $u\in \R^n$, we denote by
\[N_\Omega(\xb;u):=\Limsup_{\AT{t\downarrow 0}{u'\to u }}\widehat N_\Omega(\xb+tu')\]
the {\em directional limiting normal cone} to $\Omega$ in direction $u$ at $\xb$.

If $\xb \notin \Omega$, we put $T_{\Omega}(\xb)=\emptyset$, $\widehat N_{\Omega}(\xb)=\emptyset$, $N_{\Omega}(\xb)=\emptyset$ and $N_\Omega(\xb;u)=\emptyset$. Further note that
$N_\Omega(\xb;u)=\emptyset$ whenever $u\not\in T_\Omega(\xb)$.
If $\Omega$ is convex, then $\widehat N_\Omega(\xb)= N_\Omega(\xb)$ amounts to the classical normal cone in the sense of convex analysis and we will  write $N_\Omega(\xb)$.

Given a pair $(\xb,\xb^*)\in\gph \widehat N_\Omega$ we denote by
\[\K_\Omega(\xb,\xb^*):=T_\Omega(\xb)\cap [\xb^*]^\perp\]
the {\em critical cone} to $\Omega$ at $\xb$ with respect to $\xb^*$.

The following generalized derivatives of set-valued mappings are defined by means of the tangent cone and the (directional) limiting normal cone to the graph of the mapping.
\begin{definition}\label{DefGenDeriv}
Let $F:\R^n\tto\R^m$ be a set-valued mapping having locally closed graph around $(\xb,\yb)\in\gph F$.
\begin{enumerate}
\item [(i)]
The set-valued map $D F(\xb,\yb):\R^n\tto\R^m$, defined by
\[
DF(\xb,\yb)(u):= \{v\in\R^m\mv (u,v)\in T_{\gph F}(\xb,\yb)\}, u\in\R^n
\]
is called the {\em graphical derivative} of $F$ at $(\xb,\yb)$.
\item[(ii)]
 The set-valued map $\widehat D^\ast F(\xb,\yb ):\R^m\tto\R^n$
 \[\widehat D^\ast F(\xb,\yb)(v^\ast):=\{u^\ast\in \R^n \mv (u^\ast,- v^\ast)\in \widehat N_{\gph F}(\xb,\yb )\}, v^\ast\in \R^m\]
is called the {\em regular (Fr\'echet) coderivative} of $F$ at $(\xb,\yb )$.
\item [(iii)]
 The set-valued map $D^\ast F(\xb,\yb ): \R^m\tto\R^n$,
 defined by
\[ D^\ast F(\xb,\yb)(v^\ast):=\{u^\ast\in \R^n \mv (u^\ast,- v^\ast)\in N_{\gph F}(\xb,\yb )\}, v^\ast\in \R^m\]
is called the {\em limiting (Mordukhovich) coderivative} of $F$ at $(\xb,\yb )$.
\item [(iv)]
 Given a pair of directions $(u,v)\in\R^n\times\R^m$, the
 set-valued map  $D^\ast F((\xb,\yb ); (u,v)):\R^m\tto\R^n$, defined by
\begin{equation*}
D^\ast  F((\xb,\yb ); (u,v))(v^\ast):=\{u^\ast \in \R^n \mv (u^\ast,-v^\ast)\in N_{\gph F}((\xb,\yb ); (u,v)) \}, v^\ast\in \R^m
\end{equation*}
is called the {\em directional limiting coderivative} of $F$ in direction $(u,v)$ at $(\xb,\yb)$.
\end{enumerate}
\end{definition}

\subsection{Regularity and Lipschitzian properties of set-valued mappings}
First we recall some well-known definitions.
\begin{definition}\label{DefMetrReg}Let $F:\R^n\tto\R^m$ be a mapping and let $(\xb,\yb)\in \gph F$.
We say that $F$ is {\em metrically regular around $(\xb,\yb)$} if there are neighborhoods $U$ of $\xb$ and $V$ of $\yb$ along with some real $\kappa\geq 0$ such that
\begin{equation}\label{EqMetrReg}\dist{x,F^{-1}(y)}\leq \kappa\dist{y,F(x)}\ \forall (x,y)\in U\times V.\end{equation}
When fixing $y=\yb$ in this condition, $F$ is said to be {\em metrically subregular at $(\xb,\yb)$}, i.e., we require
\begin{equation}\label{EqSubReg}\dist{x,F^{-1}(\yb)}\leq \kappa\dist{\yb,F(x)}\ \forall x\in U.\end{equation}
\end{definition}
A well-known coderivative characterization of metric regularity is known as ''Mordukhovich criterion'' and reads as follows.
\begin{theorem}[{\cite[Theorem 3.3]{Mo18}}]\label{ThMetrReg}Assume that the set-valued mapping $F:\R^n\tto\R^m$ has locally closed graph around $(\xb,\yb)\in\gph F$. Then $F$ is metrically regular around $(\xb,\yb)$ if and only if
\begin{equation}\label{EqCondMetrReg}
 0\in D^*F(\xb,\yb)(v^*)\ \Rightarrow\ v^*=0.
\end{equation}
\end{theorem}
One can find numerous sufficient conditions for metric subregularity in the literature, see, e.g., \cite{FabHenKruOut10,FabHenKruOut12,Gfr11,Gfr13a,Gfr13b,Gfr14a,HenJouOut02, IofOut08,Kru15a,ZheNg10}. However, these sufficient conditions are often very difficult to verify. The following sufficient condition for metric subregularity is not as week as possible but it is stable with respect to certain perturbations, cf. \cite{DoGfrKruOut18}.
\begin{theorem}[{\cite[Theorem 2.6]{GfrOut16}}]\label{ThSubReg}
Assume that the set-valued mapping $F:\R^n\tto\R^m$ has locally closed graph around $(\xb,\yb)\in\gph F$. If
  \begin{equation*}
 \forall 0\not=u\in\R^n:\ 0\in D^\ast F\big((\xb,\yb);(u,0)\big)(v^\ast)\Rightarrow  v^\ast=0,
\end{equation*}
then $F$ is metrically subregular at $(\xb,\yb)$.
\end{theorem}
In order to define a directional version of metric (sub)regularity, consider for a direction $u\in\R^n$ and positive reals $\rho,\delta$ the set
\[{\cal V}_{\rho,\delta}(u):=\left\{d\in\rho\B_{\R^n}\mv \bnorm{\norm{u}d-\norm{d}u} \leq \delta \norm{u}\norm{d}\right\}.\]
We say that ${\cal V}$ is a {\em directional neighborhood of $u$} if ${\cal V}_{\rho,\delta}(u)\subset {\cal V}$ for some $\rho,\delta>0$.
\begin{definition}\label{DefDirMetrReg}
Let $F:\R^n\tto\R^m$ be a mapping and let $(\xb,\yb)\in \gph F$.
\begin{enumerate}
\item Given a direction $u\in\R^n$ we say that {\em $F$ is metrically subregular in direction $u$ at $(\xb,\yb)$} if \eqref{EqSubReg} holds with $\xb+{\cal U}$  in place of $U$, where ${\cal U}$ is a  directional neighborhood of $u$.
\item Given a direction $(u,v)\in\R^n\times\R^m$ we say that {\em $F$ is metrically regular in direction $(u,v)$ at $(\xb,\yb)$} if there is a directional neighborhoods ${\cal W}$ of $(u,v)$  together with  reals $\kappa\geq 0$ and $\delta> 0$ such that
     \eqref{EqMetrReg} holds for all $(x,y)\in (\xb,\yb)+{\cal W}$ satisfying $\norm{(u,v)}\dist{(x,y),\gph F}\leq \delta\norm{(u,v)}\norm{(x,y)-(\xb,\yb)}$.
\end{enumerate}
\end{definition}
If a mapping $F$ is metrically regular in  direction $(u,0)$ at $(\xb,\yb)$ then it is also metrically subregular in direction $u$, cf. \cite[Lemma 1]{Gfr13a}. Further note that a mapping is always metrically regular in a direction $(u,v)$ at $(\xb,\yb)$ whenever $(u,v)\not\in T_{\gph F}(\xb,\yb)$, i.e., $v\not\in DF(\xb,\yb)(u)$. Similarly, if $0\not\in DF(\xb,\yb)(u)$, then $F$ is metrically subregular in direction $u$ at $(\xb,\yb)$.
\begin{theorem}\label{ThDirMetrReg}Assume that the set-valued mapping $F:\R^n\tto\R^m$ has locally closed graph around $(\xb,\yb)\in\gph F$ and let $u\in\R^n$ be given. Then $F$ is metrically regular in direction $(u,0)$ at $(\xb,\yb)$ if and only if
\begin{equation}\label{EqCondDirMetrReg}
 0\in D^*F\big((\xb,\yb); (u,0)\big)(v^*)\ \Rightarrow\  v^*=0.
\end{equation}
\end{theorem}
\begin{proof}
Follows from \cite[Theorem 5]{Gfr13a}.
\end{proof}
Comparing  Definition \ref{DefDirMetrReg} with Definition \ref{DefMetrReg} we see that metric regularity around $(\xb,\yb)$ is equivalent with metric regularity in direction $(0,0)$ at $(\xb,\yb)$. This is reflected also in conditions \eqref{EqCondMetrReg} and \eqref{EqCondDirMetrReg} with $u=0$. Further note that the sufficient condition for metric subregularity of Theorem \ref{ThSubReg} says that mapping $F$ is metrically regular at $(\xb,\yb)$ in every direction $(u,0)$ with $u\not=0$.

The following notion of stability  was introduced by Robinson \cite{Rob76}.
\begin{definition}
Consider the system
\begin{equation}\label{EqInequSystem}h(p,x)\in C\end{equation}
for a mapping $h:P\times \R^n\to \R^m$ and a set $C\subset\R^m$, where $P$ is a topological space and denote
\[S(p):=\{x\in\R^n\mv h(p,x)\in C\},\ p\in P.\]
We say that the system \eqref{EqInequSystem} enjoys the {\em Robinson stability property} at $(\pb,\xb)\in\gph S$ if there are neighborhoods $Q$ of $\pb$, $U$ of $\xb$ and a real $\kappa\geq 0$ such that
\[\dist{x,S(p)}\leq \kappa \dist{h(p,x),C}\ \forall (p,x)\in Q\times U.\]
\end{definition}
Comparing the definition of Robinson stability with that of metric regularity we see that in case when $P=\R^l$ and $h$ is of the form $h(p,x)=\tilde h(x)-p$, the property of Robinson stability of \eqref{EqInequSystem} at $(\pb,\xb)$ is equivalent to metric regularity of the mapping $\tilde h(\cdot)-C$ around $(\xb,\pb)$.
For sufficient conditions for Robinson stability we refer to the recent paper \cite{GfrMo17a}. Here we mention only  the following result.
\begin{theorem}\label{ThRS}Let $(\pb,\xb)\in h^{-1}(C)$ be given and assume that $h$ is differentiable with respect to the second component and both $h$ and $\nabla_2 h$ are continuous, whereas $C$ is closed.  If
\[\nabla_2h(\pb,\xb)^T\mu=0, \mu\in N_C(h(\pb,\xb))\ \Rightarrow\ \mu=0,\]
then the system \eqref{EqInequSystem} enjoys the Robinson stability property at $(\pb,\xb)$.
\end{theorem}
\begin{proof}
  Follows immediately from \cite[Corollary 3.6]{GfrMo17a}.
\end{proof}
We now turn to Lipschitzian properties of set-valued mappings.
\begin{definition}
  Let $S:\R^m\tto\R^n$ be a set-valued map and let $(\yb,\xb)\in\gph S$.
  \begin{enumerate}
  \item $S$ is called to be {\em calm} at $(\yb,\xb)$ if there is a neighborhood $U$ of $\xb$ together with a real $L\geq 0$ such that
  \[S(y)\cap U\subset S(\yb)+L\norm{y-\yb}\B_{\R^n}\ \forall y\in \R^m.\]
  If, in addition, $S(\yb)=\{\xb\}$ is a singleton we say that 
  $S$ has the {\em isolated calmness property} at $(\yb,\xb)$.
  \item
  Given a set $Y\subset\R^m$ containing $\yb$, the mapping $S$ is said to have the {\em Aubin property relative to $Y$ around $(\yb,\xb)$} if there are neighborhoods $V$ of $\yb$, $U$ of $\xb$ and a real $L\geq 0$ such that
  \[S(y)\cap U\subset S(y')+L\norm{y-y'}\B_{\R^n}\ \forall y,y'\in Y\cap V.\]
  This condition with $V$ in place of $Y\cap V$ is simply the {\em Aubin propery around $(\yb,\xb)$}.
  \end{enumerate}
\end{definition}
It is well-known \cite{DoRo04} that $F$ is metrically subregular at $(\xb,\yb)$ if and only if its inverse mapping $F^{-1}$ is calm at $(\yb,\xb)$. Further,  metric regularity is equivalent with the Aubin property of the inverse mapping.

\subsection{Polyhedral sets}

Recall that a set $D\subset \R^s$ is said to be {\em convex polyhedral} if it can be represented as the intersection of finitely many halfspaces. We say that a set $E\subset \R^s$ is polyhedral if it is the union of finitely many convex polyhedral sets. If a set $E$ is polyhedral, then for every $\zb\in E$ there is some neighborhood $W$ of $\zb$ such that
\[(E-\zb)\cap W= T_E(\zb)\cap W.\]
Given a convex polyhedral set $D$ and a point $\zb\in D$, then the tangent cone $T_D(\zb)$ and the normal cone $N_D(\zb)$ are convex polyhedral cones and there is a neighborhood $W$ of $\zb$ such that
\[ T_D(z)=T_D(\zb)+[z-\zb]\supset T_D(\zb),\ N_D(z)=N_D(\zb)\cap [z-\zb]^\perp\subset N_D(\zb)\ \forall z\in W.\]
The graph of the normal cone mapping to $D$ is a polyhedral set and for every pair $(z,z^*)\in\gph N_D$ we have
\begin{equation}\label{EqTangConeGraphNormalCone} T_{\gph N_D}(z,z^*)=\gph N_{\K_D(z,z^*)},\end{equation}
see, e.g., \cite[Lemma 2E.4]{DoRo14}.

For two convex polyhedral cones $K_1,K_2\subset\R^s$ their polars as well as their sum $K_1+K_2$ and their intersection $K_1\cap K_2$ are again convex polyhedral cones and
\[(K_1+K_2)^\circ =K_1^\circ\cap K_2^\circ,\ (K_1\cap K_2)^\circ = K_1^\circ+ K_2^\circ.\]
For a convex polyhedral cone $K\subset\R^s$ and a point $z\in K$  we have
\[T_K(z)=K+[z],\ N_K(z)=K^\circ\cap [z]^\perp.\]
A face $\F$ of  $K$ can always be written in the form
\[\F=K\cap [z^*]^\perp\]
for some $z^*\in K^\circ$. The cone $K$ has the representation
\begin{equation}\label{EqReprPolyCone}K=\left\{z\in\R^s\mv a_i^Tz=0,\ i\in \bar J,\ a_i^Tz\leq 0\ i\in\bar I\setminus\bar J\right\},\end {equation}
where $\bar J\subset \bar I$ are two finite index sets and $a_i\in\R^s$, $i\in \bar I$. By enlarging $\bar J$ when necessary we can assume that there exists some $z_0$ such that $a_i^Tz_0=0$, $i\in\bar J$, $a_i^Tz_0<0$, $i\in \bar I\setminus\bar J$. Then a subset $\F\subset K$ is a face if and only if there is some index set $J$, $\bar J\subset J\subset\bar I$ such that
\[\F=\left\{z\in\R^s\mv  a_i^Tz=0,\ i\in  J,\ a_i^Tz\leq 0\ i\in\bar I\setminus J\right\}.\]
By possibly enlarging $J$ we can find a unique index set, denoted by $J_\F$, such that
\begin{equation}\label{EqFaceIndexSet}
\ri \F=\left\{z\in\R^s\mv  a_i^Tz=0,\ i\in  J_\F,\ a_i^Tz< 0\ i\in\bar I\setminus J_\F\right\}.\end{equation}
It follows that
\[\F-\F=\{z\in\R^s\mv a_i^Tz=0,\ i\in J_\F\}.\]

\subsection{Directional non-degeneracy}
In what follows the property of directional metric (sub)regularity of a particular mapping will play an important role.
Let $D\subset \R^s$ be a convex polyhedral set, let $\tilde g: \R^m \to \R^s$ be  continuously differentiable and consider
the mapping $F:\R^m\times\R^s\tto\R^s\times\R^s$ given by
\begin{equation} \label{EqTildeF}
F(y,\lambda):=\big(\tilde g(y), \lambda \big)-\gph N_{D}.
\end{equation}
Given some point $(\yb,\lambda)\in F^{-1}(0)$ and some direction $(v,\eta)\in\R^m\times\R^s$, we we want to investigate metric subregularity of $F$ in direction $(v,\eta)$ at $(\yb,\lambda)$, in particular when $v\not=0$.
We denote
\[\Theta(\yb,v):=\{(\lambda,\eta)\in N_D(\tilde g(\yb))\times\R^s\mv (\nabla \tilde g(\yb)v,\eta)\in \gph N_{\K_{D}(\tilde g(\yb),\lambda)}\},\ (\yb,v)\in\tilde g^{-1}(D)\times\R^m.\]
Recall that $F$ is by definition metrically subregular in direction $(v,\eta)$ at $(\yb,\lambda)$ whenever
\[(0,0)\not\in DF\big((\yb,\lambda),(0,0)\big)(v,\eta)=(\nabla \tilde g(\yb)v,\eta)- T_{\gph N_{D}}(\tilde g(\yb),\lambda)\ \Leftrightarrow\ (\nabla \tilde g(\yb)v,\eta)\not\in T_{\gph N_{D}}(\tilde g(\yb),\lambda),\]
i.e., taking into account \eqref{EqTangConeGraphNormalCone}, whenever $(\lambda,\eta)\not\in\Theta(\yb,v)$.

In our analysis we restrict ourselves to the  characterization of  metric regularity of $F$ in directions $\big((v,\eta),(0,0)\big)$ which implies metric subregularity of $F$ in direction $(v,\eta)$. The following lemma is a slight generalization of \cite[Proposition 2]{GfrOut18}.
\begin{lemma}\label{LemF_DirMetrReg}Let $\yb\in\tilde g^{-1}(D)$, $v\in\R^m$  and $(\lambda,\eta)\in\Theta(\yb,v)$ be given. Then the mapping $F$ defined in \eqref{EqTildeF} is metrically regular in direction $\big((v,\eta),(0,0)\big)$ at $\big((\yb,\lambda),(0,0)\big)$ if and only if for every face $\F$ of the critical cone $\K_{D}\big(\tilde g( \yb), \lambda\big)$ with $\nabla \tilde g( \yb)(v)\in \F\subset[\eta]^\perp$ one has
  \begin{equation*}
    \nabla \tilde g(\yb)^T\mu=0,\ \mu\in (\F-\F)^\circ\ \Rightarrow\ \mu=0.
  \end{equation*}
\end{lemma}
\begin{proof}
  The characterization  \eqref{EqCondDirMetrReg} reads in our special case as
  \[(\nabla \tilde g(\yb)^T\mu,\xi)=(0,0),\ (\mu,\xi)\in N_{\gph N_{D}}\big((\tilde g(\yb),\lambda), (\nabla \tilde g(\yb)v,\eta)\big)\ \Rightarrow\ (\mu,\xi)=(0,0),\]
  see also \cite[Theorem 1]{GfrKl16}. By \cite[Theorem 2.12]{GfrOut16}, $N_{\gph N_{D}}\big((\tilde g(\yb),\lambda), (\nabla \tilde g(\yb)v,\eta)\big)$ amounts to the union of all product sets $K^\circ\times K$ associated with cones $K$ of the form $\F_1-\F_2$, where $\F_1,\F_2$ are faces of the critical cone
$\K_{D}\big(\tilde g(\yb),\lambda\big)$ with $\nabla \tilde g(\yb)(v)\in \F_2\subset \F_1\subset[\eta]^\perp$. Thus, by Theorem \ref{ThDirMetrReg} the claimed directional metric regularity is equivalent to the condition that the implication
\[\nabla \tilde g(\yb)^T\mu=0,\ \mu\in (\F_1-\F_2)^\circ\ \Rightarrow\ \mu=0\]
holds for every pair of faces $\F_1,\F_2$  with $\nabla \tilde g(\yb)(v)\in \F_2\subset \F_1\subset[\eta]^\perp$. By taking into account that $(\F_1-\F_2)^\circ\subset (\F_2-\F_2)^\circ$, the statement of the lemma follows.
\end{proof}
This characterization of directional metric regularity can be considerably simplified.
\begin{theorem}\label{ThDirNonDegen}
  Let $\yb\in\tilde g^{-1}(D)$ and $v\in\R^m$ be given and assume that $\Theta(\yb,v)\not=\emptyset$. Then the following statements are equivalent:
  \begin{enumerate}
    \item There is some $(\lb,\bar \eta)\in\Theta(\yb,v)$ such that the mapping $F$ given by \eqref{EqTildeF} is metrically regular in direction $\big((v,\bar \eta),(0,0)\big)$ at $\big((\yb,\lb),(0,0)\big)$.
    \item The mapping $F$ given by \eqref{EqTildeF} is metrically regular in direction $\big((v,\eta),(0,0)\big)$ at $\big((\yb,\lambda),(0,0)\big)$ for every $(\lambda,\eta)\in\Theta(\yb,v)$.
    \item
   \begin{equation}\label{EqDirNonDegen}\nabla\tilde g(\yb)^T\mu=0,\ \mu\in\Span N_{T_D(\tilde g(\yb))}(\nabla \tilde g(\yb)v)\ \Rightarrow\ \mu=0.\end{equation}
  \end{enumerate}
\end{theorem}
\begin{proof}Assume that the tangent cone $T_D(\tilde g(\yb))$ has the representation \eqref{EqReprPolyCone}
 and consider any $(\lambda,\eta)\in\Theta(\yb,v)$. Since $\eta\in N_{\K_D(\tilde g(\yb),\lambda)}(\nabla \tilde g(\yb)v)=\K_D(\tilde g(\yb),\lambda)^\circ\cap [\nabla \tilde g(\yb)v]^\perp$,  $\nabla \tilde g(\yb)v$ is contained in the face $\K_D(\tilde g(\yb),\lambda)$ of $T_D(\tilde g(\yb))$ and therefore
\[\J(v):=\{i\in \bar I\mv a_i^T\nabla \tilde g(y)v=0\}\supset J_\lambda:=J_{\K_D(\tilde g(\yb),\lambda)},\]
where $J_{\K_D(\tilde g(\yb),\lambda)}$ is given by \eqref{EqFaceIndexSet}.
Further, $\eta$ has the representation $\eta=\sum_{i\in I_\eta}a_i^T\sigma_i$ with $J_\lambda\subset I_\eta\subset \J(v)$ and $\sigma_i>0$, $i\in I_\eta\setminus J_\lambda$. Next consider any face $\F$ of the critical cone $\K_D(\tilde g(\yb),\lambda)$ satisfying $\nabla \tilde g(\yb)v\in\F\subset[\eta]^\perp$. Then $\F$ is again a face of $T_D(\tilde g(y))$ and from  $\nabla \tilde g(\yb)v\in\F\subset[\eta]^\perp$ we deduce
\[I_\eta\subset J_\F\subset \J(v).\]
Thus
\[\F\supset \F_v:=\left\{z\mv a_i^Tz=0,\ i\in\J(v),\ a_i^Tz\leq 0,\ i\in \bar I\setminus \J(v)\right\}\]
and therefore $(\F-\F)^\circ\subset (\F_v-\F_v)^\circ$. Since $\F_v$ is also a face of $\K_D(\tilde g(\yb),\lambda)$ satisfying $\nabla \tilde g(\yb)v\in\F_v\subset[\eta]^\perp$, by Lemma \ref{LemF_DirMetrReg} $F$ is metrically regular in direction $\big((v,\eta),(0,0)\big)$ at $\big((\yb,\lambda),(0,0)\big)$ if and only if
  \begin{equation}
    \label{EqDirMetrRegSimpl}\nabla \tilde g(\yb)^T\mu=0,\ \mu\in (\F_v-\F_v)^\circ\ \Rightarrow\ \mu=0.
  \end{equation}
Since $\F_v$ depends neither on $\lambda$ nor on $\eta$, the equivalence between (i) and (ii) is established. To show the equivalence of \eqref{EqDirMetrRegSimpl} with \eqref{EqDirNonDegen} just observe that $\F_v-\F_v=\{z\mv a_i^Tz=0, i\in \J(v)\}$ implying $(\F_v-\F_v)^\circ=\{\sum_{i\in\J(v)}\sigma_ia_i\mv \sigma_i\in\R, i\in \J(v)\}$ and
$N_{T_D(\tilde g(\yb))}(\nabla \tilde g(\yb)v)=\{\sum_{i\in\J(v)}\sigma_ia_i\mv \sigma_i\geq 0, i\in \J(v)\}$. Thus  $\Span N_{T_D(\tilde g(\yb))}(\nabla \tilde g(\yb)v)=(\F_v-\F_v)^\circ$ and the proof is complete.
\end{proof}
From the proof of Theorem \ref{ThDirNonDegen} we also obtain  the following corollary.
\begin{corollary}\label{CorUnionFaces}
   Let $\yb\in\tilde g^{-1}(D)$, $\lambda\in N_D(\tilde g(\yb))$, $v\in\R^l$ and $\eta\in N_{\K_D(\tilde g(\yb),\lambda)}(\nabla \tilde g(\yb)v)$ be given. Then the union of all sets $(\F-\F)^\circ$, where $\F$ is a face of the critical cone $\K_D(\tilde g(\yb),\lambda)$ satisfying $\nabla \tilde g(\yb)v\subset\F\subset[\eta]^\perp$, is exactly $\Span N_{T_D(\tilde g(\yb))}(\nabla \tilde g(\yb)v)$.
\end{corollary}
\begin{definition}\label{DefDirNonDegen} Let $\yb\in\tilde g^{-1}(D)$ and $v\in\R^m$ be given.
 We say that the system $\tilde g(\cdot)\in D$ is {\em non-degenerate in direction $v$ at $\yb$} if condition \eqref{EqDirNonDegen} is fulfilled.
  In case when $v=0$ we simply say that the system $\tilde g(\cdot)\in D$ is non-degenerate at $\yb$.
\end{definition}
Note that \eqref{EqDirNonDegen} is automatically fulfilled if $\nabla \tilde g(\yb)v\not\in T_D(\tilde g(\yb))$. Further, if $\nabla \tilde g(\yb)v\in T_D(\tilde g(\yb))$, then \eqref{EqDirNonDegen} is equivalent to
\begin{eqnarray*}\R^s&=&\{0\}^\perp=\Big(\ker \nabla\tilde g(\yb)^T\cap \Span N_{T_D(\tilde g(\yb))}(\nabla \tilde g(\yb)v))\Big)^\perp=\nabla\tilde g(\yb)\R^m+ \Big(\Span N_{T_D(\tilde g(\yb))}(\nabla \tilde g(\yb)v)\Big)^\perp\\
&=&\nabla\tilde g(\yb)\R^m+ \Big(N_{T_D(\tilde g(\yb))}(\nabla \tilde g(\yb)v)-N_{T_D(\tilde g(\yb))}(\nabla \tilde g(\yb)v)\Big)^\circ\\
&=&\nabla\tilde g(\yb)\R^m+ \big(N_{T_D(\tilde g(\yb))}(\nabla \tilde g(\yb)v)\big)^\circ\cap\big(-N_{T_D(\tilde g(\yb))}(\nabla \tilde g(\yb)v)\big)^\circ\end{eqnarray*}
which in turn is equivalent to
\begin{equation}\label{EqDirNonDegenPrimal}
\nabla\tilde g(\yb)\R^m+\lin T_{T_D(\tilde g(\yb))}(\nabla \tilde g(\yb)v)=\R^s.
\end{equation}
Clearly, for $v=0$ we obtain the standard definition of non-degeneracy from \cite[Formula 4.17]{BonSh00}.

We now state some properties of directional non-degeneracy.
\begin{proposition}\label{PropDirNonDegen}
  Let $\yb\in \tilde g^{-1}(D)$ and $v\in\R^n$ such that the system $\tilde g(\cdot)\in D$ is non-degenerate in direction $v$ at $\yb$. Then
  there is a directional neighborhood ${\cal V}$ of $v$ and a constant $\beta>0$ such that for all $y\in \big((\yb+{\cal V})\cap \tilde g^{-1}(D)\big)$, $y\not=\yb$, one has
  \begin{equation}\label{EqUniformNonDegen}
    \norm{\nabla \tilde g(y)^T\mu}\geq \beta\norm{\mu}\ \forall \mu\in \Span N_D(\tilde g(y)).
  \end{equation}
  In particular, for all $y\in \big((\yb+{\cal V})\cap \tilde g^{-1}(D)\big)$, $y\not=\yb$, the system $\tilde g(\cdot)\in D$ in non-degenerate at $y$.
\end{proposition}
\begin{proof}By contraposition. Assume on the contrary that there are sequences $t_k\downarrow 0$, $y_k\in \tilde g^{-1}(D)$, $\mu_k\in (\Span N_D(\tilde g(y_k))\cap \Sp_{\R^s}$ such that $\lim_{k\to\infty}\nabla \tilde g(y_k)^T\mu_k=0$ and $\lim_{k\to\infty}(y_k-\yb)/t_k \to v$. Since for all $k$ sufficiently large we have
\begin{eqnarray*}N_D(\tilde g(y_k))&=& (T_D(\tilde g(\yb))^\circ\cap[\tilde g(y_k)-\tilde g(\yb)]^\perp=
\left(T_D(\tilde g(\yb))+\left[\frac{\tilde g(y_k)-\tilde g(y)}{t_k}\right]\right)^\circ\\
&=&\left(T_{T_D(\tilde g(y))}\left(\frac{\tilde g(y_k)-\tilde g(y)}{t_k}\right)\right)^\circ=N_{T_D(\tilde g(y))}\left(\frac{\tilde g(y_k)-\tilde g(y)}{t_k}\right)\\
&=&N_{T_D(\tilde g(y))}(\nabla \tilde g(y)v)\cap \left[\frac{\tilde g(y_k)-\tilde g(y)}{t_k}-\nabla \tilde g(y)v\right]^\perp\subset N_{T_D(\tilde g(y))}(\nabla \tilde g(y)v),\end{eqnarray*}
it holds that $\mu_k\in \Span N_{T_D(\tilde g(y))}(\nabla \tilde g(y)v)$ and, by passing to some subsequence if necessary, we can assume that $\mu_k$ converges to some $\mu\in \Big(\Span N_{T_D(\tilde g(y))}(\nabla \tilde g(y)v)\Big)\cap\Sp_{\R^s}$. Obviously we also have $\nabla g(y)^T\mu=0$, a contradiction to the assumed directional non-degeneracy and \eqref{EqUniformNonDegen} is proved. The additional statement concerning the non-degeneracy is an immediate consequence of \eqref{EqUniformNonDegen}.
\end{proof}
\if{
\begin{proposition}Let $\yb\in \tilde g^{-1}(D)$ and assume that the system $\tilde g(\cdot)-D$ is non-degenerate in every direction $0\not=v\in\R^m$ at $\yb$. Then for every $\lambda\in N_D(\tilde g(\yb)$ the mapping $F$ given by \eqref{EqTildeF} is metrically subregular at $(\yb,\lambda)$.
\end{proposition}
\begin{proof}
  t.b.d.
\end{proof}}\fi

It turns out that the directional non-degeneracy can be fulfilled
in all non-zero directions even if the (standard) non-degeneracy is violated.
\begin{example}
  Let $D=\R^s_-$ and assume $\tilde g(\yb)=0$. Given a direction $v$ satisfying $\nabla \tilde g(\yb)v\leq 0$, we have $\Span N_{T_D(\tilde g(\yb))}(\nabla \tilde g(\yb)v)=\{\mu\in\R^s\mv\mu_i=0,\ i\not\in \J(v)\}$, where $\J(v):=\{i\mv \nabla \tilde g_i(\yb)v=0\}$. Thus, non-degeneracy in direction $v$ is equivalent to the linear independence of the gradients $\nabla \tilde g_i(\yb)$, $i\in \J(v)$ whereas non-degeneracy amounts to the so-called {\em linear independence constraint qualification} (LICQ), i.e., to the linear independence of all gradients $\nabla \tilde g_i(\yb)$, $i=1,\ldots,s$.

  Consider  the system
  \[y_1-y_4\leq 0,\ -y_1-y_4\leq 0,\ y_2-y_4\leq 0,\ -y_2-y_4\leq 0,\ y_3+y_1^2-y_4\leq 0,\ -y_3-y_4\leq 0.\]
  Obviously LICQ is violated at $\yb=0$. However, it is not difficult to verify that the system is non-degenerate in every direction $v\not=0$.\\
  Further note that in this example also the so-called {\em constant rank constraint qualification} is violated at $\yb$.\hfill$\triangle$
\end{example}

\if{\begin{proposition}\label{PropMSvsMaxMult}
Assume we are given $y\in\R^n$, $\lambda\in\R^s$, $v\in\R^n$ and $\eta\in \R^s$ with $\lambda \in N_D(\tilde g(y))$, $(\nabla \tilde g(y)v,\eta)\in \gph N_{\K_{D}(\tilde g(y),\lambda)}$ and assume that the mapping
\eqref{EqTildeF}
is metrically subregular in direction $(v,\eta)$ at $\big((y,\lambda),(0,0)\big)$.
Then
\[\lambda \in \argmax\{v^T\nabla^2\skalp{\mu^T \tilde g}(\tilde y)v \mv \mu\in N_{D}(\tilde g(y)), \nabla \tilde g(y)^T \mu=y^*\},
\]
where $y^*:=\nabla\tilde g(y)^T\lambda$.
\end{proposition}
\begin{proof}
Since $(\nabla \tilde g(y)v,\eta)$ belongs to the polyhedral set $\gph N_{\K_{D}(\tilde g(y),\lambda)}=T_{\gph N_{D}}(\tilde g(y),\lambda)$, we have $(\tilde g(y)+t\nabla \tilde g(y)v,\lambda+t\eta)\in \gph N_{D}$ for all $t\geq 0$ sufficiently small and thus
\[\dist{(\tilde g(y + t v), \lambda + t \eta),\gph N_{D}} =\oo(t).\]
By the assumed metric subregularity we obtain for $t\geq 0$ the existence of $y_t$
and $\lambda_t$ such that $(\tilde g(y_t),\lambda_t)\in\gph N_{D}$ and
\[\lim_{t\downarrow 0}\norm{(y_t -(y+ tv), \lambda_t - (\lambda + t\eta))}/t =0,\]
i.e., $v_t := (y_t - \tilde y)/t \to v$ and
$\eta_t := (\lambda_t - \lambda)/t \to \eta$ as $t\downarrow 0$.

Consider arbitrary $\mu \in N_{D}(\tilde g(y))$
with $\nabla \tilde g(y)^T  \mu =y^*=\nabla \tilde g(y)^T \lambda$.
By convexity of $D$, together with
$\lambda+t\eta_t\in)N_{D}(\tilde g(y + t v_t))$, we conclude
$ \mu^T (\tilde g(y + t v_t) - \tilde g(y)) \leq 0$,
$(\lambda + t \eta_t)^T (\tilde g(y)-\tilde g(y + t v_t) ) \leq 0$ and consequently, by taking into account $(\mu-\lambda)^T\nabla g(y)=0$ and $\eta^T\nabla g(y)v=0$,
\begin{eqnarray*}
 0 & \geq & \limsup_{t\downarrow 0}\big(\mu - (\lambda + t \eta_t)\big)^T \big(\tilde g(y + t v_t) - \tilde g( y)\big)/t^2 \\
 &=&\limsup_{t\downarrow 0}\big(\mu - (\lambda + t \eta_t)\big)^T\big(\nabla \tilde g(y)v_t/t +\frac 12 v_t^T\nabla^2\tilde g(y)v_t\big)\\
 &=&\limsup_{t\downarrow 0}\Big(\eta_t^T\nabla \tilde g(y)v_t+\frac 12 v_t^T\nabla^2\big((\mu-\lambda)^T\tilde g\big)(y)v_t\Big)
 =\frac 12 v^T\nabla^2\big((\mu-\lambda)^T\tilde g\big)(y)v,
\end{eqnarray*}
and the assertion follows.
\end{proof}}\fi

\section{Stability properties through generalized differentiation}

Throughout this section we consider the solution mapping $S$ given by  \eqref{EqGeneralSolMap}. Given some reference point $(\pb,\xb)\in \gph S$, we will provide point-based sufficient conditions for the isolated calmness property, the Aubin property and the Aubin property relative to some set $P\subset\R^l$, respectively,  in terms of generalized derivatives of the mapping $M$.

We start with the Levy-Rockafellar characterization of isolated calmness \cite{Lev96}, who showed that
\begin{equation}\label{EqLevy}\mbox{$S$ is isolated calm at $(\pb,\xb)\in\gph S$}\ \Leftrightarrow DS(\pb,\xb)(0)=\{0\}.\end{equation}
\begin{theorem}\label{ThIsolCalmness}
  Assume that $M$ has locally closed graph around the reference point $(\pb,\xb,0)\in\gph M$. If
  \begin{equation}\label{EqSuffCondIsolCalm}
  0\in DM(\pb,\xb,0)(0,u)\ \Rightarrow\ u=0,
  \end{equation}
  then $S$ has the isolated calmness property at $(\pb,\xb)$. Conversely, if there is some $u\not=0$ such that $0\in DM(\pb,\xb,0)(0,u)$ and $M$ is  metrically subregular in direction $(0,u)$ then $S$ is not isolatedly calm at $(\pb,\xb)$.
\end{theorem}
\begin{proof}Note that the closedness of $\gph M$ readily implies that $\gph S=M^{-1}(0)$ is locally closed around $(\pb,\xb)$. The sufficiency of \eqref{EqSuffCondIsolCalm} for the isolated calmness property of $S$ is due to \eqref{EqLevy} together with the inclusion
\begin{equation*}
DS(\pb,\xb)(0)\subset \{u\mv 0\in DM(\pb,\xb,0)(0,u)\}\end{equation*}
following  from the definition of the graphical derivative, see also \cite[Theorem 3.1]{Lev96}. In order to show the second statement, consider $u\not=0$ verifying $0\in DM(\pb,\xb,0)(0,u)$ and assume that $M$ is  metrically subregular in direction $(0,u)$ at $(\pb,\xb,0)$. By \cite[Proposition 4.1]{GfrOut16} we obtain
$(0,u)\in T_{M^{-1}(0)}(\pb,\xb)=T_{\gph S}(\pb,\xb)$ and consequently $u\in DS(\pb,\xb)(0)$. Thus mapping $S$ is not isolatedly calm at $(\pb,\xb)$ by \eqref{EqLevy}.
\end{proof}
Since metric subregularity of $M$ implies metric subregularity in any direction, we obtain the following corollary.
\begin{corollary}\label{CorIsolCalm1}
  Assume that $M$  has locally closed graph around and is metrically subregular at $(\pb,\xb,0)\in\gph M$. Then $S$ is isolatedly calm at $(\pb,\xb)$ if and only if \eqref{EqSuffCondIsolCalm} holds.
\end{corollary}
A sufficient condition for the Aubin property of $S$ around $(\pb,\xb)$ is constituted by the following theorem.

\begin{theorem}[{\cite[Theorem 4.4]{GfrOut16}}]\label{ThAubin}
 Assume that $M$ has locally closed graph around the reference point $(\pb,\xb,0)\in\gph M$ and assume that
 \begin{enumerate}
  \item [(i)]
  \[  \{u \in \R^n \mv 0 \in DM(\pb,\xb,0)(q,u)\} \neq \emptyset \mbox{ for all } q \in \R^l;
  \]
   \item [(ii)]  $M$ is metrically subregular at $(\pb,\xb,0)$;
 \item [(iii)]
 For every nonzero $(q,u)\in \R^l\times\R^n$ verifying $0 \in  DM(\pb,\xb,0)(q,u)$ one has the implication
 \begin{equation*}
(q^*,0)\in D^*M ((\pb,\xb,0); (q,u,0))(v^*)\Rightarrow q^*=0.
 \end{equation*}
 \end{enumerate}
 Then $S$ has the Aubin property around $(\pb,\xb)$ and for any $q \in \R^l$
 \begin{equation*}
DS(\pb,\xb)(q)=\{u\mv 0\in DM(\pb,\xb,0)(q,u)\}.
 \end{equation*}
 The above assertions remain true provided assumptions (ii), (iii) are replaced by
 \begin{enumerate}
\item [(iv)] For every nonzero $(q,u)\in\R^l\times\R^n$ verifying $0 \in DM(\pb,\xb,0)(q,u)$ one has the implication
\begin{equation*}
(q^*,0)\in D^*M((\pb,\xb,0); (q,u,0))(v^*)\Rightarrow \left \{ \begin{array}{l} q^*=0\\ v^*=0. \end{array}\right.
 \end{equation*}
 \end{enumerate}
\end{theorem}
Sufficient conditions for the Aubin property of $S$ relative to some set $P$ are based on the following statement, where $h:P\times \R^n\to\R^l\times\R^n\times\R^m,\ h(p,x):=(p,x,0)$.
\begin{proposition}\label{PropRelAubin}
Let $(\pb,\xb,0)\in\gph M$ and consider a subset $P\subset\R^l$ containing $\pb$. If the system
\begin{equation}\label{EqSystem4RS}h(p,x)\in \gph M \end{equation}
 enjoys the Robinson stability property at $(\pb,\xb)$, where $P$ is equipped with the induced norm topology of $\R^l$, then $S$ has the Aubin property relative to $P$ around $(\pb,\xb)$.
\end{proposition}
\begin{proof}
Obviously $S$ is also the solution mapping of the inclusion $(p,x,0)\in\gph M$. By the definition of the Robinson stability together with the assumption on the topology of $P$, there are neighborhoods $Q$ of $\pb$ in $\R^l$, $U$ of $\xb$ and a constant $\kappa\geq 0$ such that
\[\dist{x,S(p)}\leq\kappa\dist{(p,x,0),\gph M}\ \forall (p,x)\in (Q\cap P)\times U.\]
Next consider $p,p'\in Q\cap P$ and $x\in S(p)\cap U$. Then
\[\dist{x,S(p')}\leq \kappa\dist{(p',x,0),\gph M}\leq \kappa\big(\dist{(p,x,0),\gph M}+\norm{p-p'}\big)=\kappa\norm{p-p'}\]
 and thus $x\in S(p')+(\kappa+1)\norm{p-p'}\B_{\R^n}$. It follows that $S(p)\cap U\subset S(p')+(\kappa+1)\norm{p-p'}\B_{\R^n}$ showing the Aubin property of $S$ relative to $P$.
\end{proof}
\begin{theorem}\label{ThRelAubin}
 Assume that $M$ has a locally closed graph around the reference point $(\pb,\xb,0)\in\gph M$ and consider a closed set $P\subset\R^l$ containing $\pb$. Further assume that
 \begin{enumerate}
 \item[(i)] for every $q\in T_P(\pb)$ and every sequence $t_k\downarrow 0$ there exists some $u\in \R^n$ satisfying
 \begin{equation}\label{EqRelAubin_Cond1}\liminf_{k\to\infty}\;\dist{(\pb+t_kq,\xb +t_k u,0),\gph M}/t_k=0\end{equation}
 \item[(ii)] For every nonzero $(q,u)\in T_P(\pb)\times\R^n$ verifying $0 \in DM(\pb,\xb,0)(q,u)$ one has the implication
\begin{equation}\label{EqRelAubin_Cond2}
(q^*,0)\in D^*M((\pb,\xb,0); (q,u,0))(v^*)\Rightarrow \left \{ \begin{array}{l} q^*=0\\ v^*=0. \end{array}\right.
 \end{equation}
  \end{enumerate}
  Then $S$ has the Aubin property relative to $P$ around $(\pb,\xb)$ and for any $q \in T_P(\pb)$
 \begin{equation}\label{EqGraphDerS}
DS(\pb,\xb)(q)=\{ u  \mv 0\in DM(\pb,\xb,0)(q,u)\}.
 \end{equation}
\end{theorem}
\begin{proof}First, we apply \cite[Corollary 3.6]{GfrMo17a} to show the Robinson stability property of the system \eqref{EqSystem4RS}
at $(\pb,\xb)$. By taking $\zeta(p)=\norm{p-\pb}$ we obtain that the image derivative ${\rm Im}_\zeta D_ph(\pb,\xb)$ defined in \cite{GfrMo17a} as the closed cone  generated by $0$ and those $v\in \R^l\times\R^n\times\R^m$  for which there is a sequence ${p_k}\subset P$ with
\begin{gather*}0<\norm{h(p_k,\xb)-h(\pb,\xb)}<k^{-1},\ \norm{\nabla_x h(p_k,\xb)-\nabla_x h(\pb,\xb)}<k^{-1},\ \vert\zeta(p_k)-\zeta(\pb)\vert<k^{-1},\\
v=\lim_{k\to\infty} \frac{h(p_k,\xb)-h(\pb,\xb)}{\norm{h(p_k,\xb)-h(\pb,\xb)}}=\lim_{k\to\infty}\frac{(p_k-\pb,0,0)}{\norm{p_k-\pb}},\end{gather*}
is exactly the set $\{(q,0,0)\mv q\in T_P(\pb)\}$. Further for every $u\in\R^n$ we have $\nabla_x h(\pb,\xb)u=(0,u,0)$ and thus  \cite[Condition 3.10]{GfrMo17a} is fulfilled by \eqref{EqRelAubin_Cond1}. Next we have to verify that for every pair $(0,0)\not=(q,u)\in T_P(\pb)\times\R^n$ satisfying $(q,u,0)\in T_{\gph M}(\pb,\xb,0)$ the implication
\[\lambda\in N_{\gph M}\big((\pb,\xb,0),(q,u,0)\big),\ \nabla_x h(\pb,\xb)^T\lambda=0\ \Rightarrow\  \lambda=0\]
is fulfilled.
Setting $\lambda:=(q^*,u^*,-v^*)$ this amounts to
\[(q^*,u^*)\in D^*M((\pb,\xb,0); (q,u,0))(v^*),\ u^*=0\ \Rightarrow\ (q^*, u^*, -v^*)=(0,0,0),\]
which is obviously equivalent to \eqref{EqRelAubin_Cond2}. By taking into account that the condition $(q,u,0)\in T_{\gph M}(\pb,\xb,0)$ is the same as requiring $0\in DM(\pb,\xb,0)(q,u)$, all assumption of \cite[Corollary 3.6]{GfrMo17a} are fulfilled and the claimed Robinson stability property of the system \eqref{EqSystem4RS} at $(\pb,\xb)$ follows. By virtue of Proposition \ref{PropRelAubin} this implies the Aubin property of $S$ relative to $P$ around $(\pb,\xb)$. There remains to show \eqref{EqGraphDerS}. Since $\{ u  \mv 0\in DM(\pb,\xb,0)(q,u)\}\supset DS(\pb,\xb)(q)$ always holds by \cite[Theorem 3.1]{Lev96}, we only have to show $\{ u  \mv 0\in DM(\pb,\xb,0)(q,u)\}\subset DS(\pb,\xb)(q)$. Consider $u$ satisfying  $0\in DM(\pb,\xb,0)(q,u)$ for some $q\in T_P(\pb)$. By Theorem \ref{ThDirMetrReg}, condition \eqref{EqRelAubin_Cond2} implies that $M$ is metrically subregular in direction $(q,u)$ at $(\pb,\xb,0)$ and hence we can invoke \cite[Proposition 4.1]{GfrOut16} to obtain $(q,u)\in T_{M^{-1}(0)}(\pb,\xb)=T_{\gph S}(\pb,\xb)$ and consequently $u\in DS(\pb,\xb)(q)$. Thus $\{ u  \mv 0\in DM(\pb,\xb,0)(q,u)\}\subset DS(\pb,\xb)(q)$ and the proof of the theorem is complete.
\end{proof}
\begin{remark}
  \label{RemDerivTangents} Assumption (i) of Theorem \ref{ThRelAubin} is fulfilled in particular if for every $q\in T_P(\pb)$ there is some $u\in\R^n$ satisfying $0\in DM(\pb,\xb,0)(q,u)$ and the tangent $(q,u,0)$ to $\gph M$ is derivable. We see that in this case  Theorem \ref{ThRelAubin} is a generalization of  Theorem \ref{ThAubin}.
\end{remark}

\section{Graphical derivative of the normal cone mapping}

This section deals with computation of the graphical derivative of $M$
given by \eqref{EqVarSystem}. Throughout the rest of the paper we assume that we are given a reference solution $(\pb,\xb)$ of \eqref{EqVarSystem} fulfilling the following assumption.
\begin{assumption}\label{AssA1}
  There is some $\kappa>0$ such that for all $(p,x,z)$ belonging to a neighborhood of $(\pb,\xb,\xb)$ the inequality
  \[\dist{z,\Gamma(p,x)} \leq \kappa \dist{g(p,x,z),D}\]
  holds.
\end{assumption}
Note that by Theorem \ref{ThRS} Assumption \ref{AssA1} is  fulfilled, e.g., in the case when
\begin{equation}\label{EqDualRCQ}\nabla_3 g(\pb,\xb,\xb)^T\mu=0,\ \mu\in N_D(g(\pb,\xb,\xb))\ \Rightarrow\ \mu=0\end{equation}
which is equivalent to {\em Robinson's constraint qualification}
\[\nabla_3g(\pb,\xb,\xb)\R^n+T_D(g(\pb,\xb,\xb))=\R^s.\]
As a consequence of Assumption \ref{AssA1} we obtain that for all $(p,x,z)\in\gph\Gamma$ sufficiently close to $(\pb,\xb,\xb)$ the mapping $g(p,x,\cdot)-D$ is metrically subregular at $(z,0)$ with modulus $\kappa$ and therefore
\[\widehat N_{\Gamma(p,x)}(z)=\nabla_3 g(p,x,z)^TN_D(g(p,x,z)).\]
Moreover, for every $z^*\in\widehat N_{\Gamma(p,x)}(z)$ there is a multiplier $\lambda\in N_D(g(p,x,z))$ with
\[z^\ast=\nabla_3 g(p,x,z)^T\lambda,\ \norm{\lambda}\leq \kappa\norm{z^*},\]
cf. \cite[Lemma 2.1]{GfrMo17a}.
Finally, since $\gph \Gamma=\{(p,x,z)\mv g(p,x,z)\in D\}$ and $\dist{(p,x,z),\gph\Gamma}\leq \dist{z,\Gamma(p,x)}$, we conclude that the mapping $g(\cdot)-D$ is metrically subregular at $\big((p,x,z),0\big)$ for every $(p,x,z)\in\gph \Gamma$ sufficiently close to $(\pb,\xb,\xb)$. Therefore
\begin{gather*}T_{\gph \Gamma}(p,x,z)=\big\{(q,u,w)\in\R^l\times\R^n\times\R^n\mv \nabla g(p,x,z)(q,u,w)\in T_D\big(g(p,x,z)\big)\big\},\\
 \widehat N_{\gph \Gamma}(p,x,z)=\nabla g(p,x,z)^TN_D\big(g(p,x,z)\big).
 \end{gather*}
In order to unburden the notation  we introduce  the mappings
\[b(p,x):=\nabla_3g(p,x,x),\quad \tilde g(p,x):=g(p,x,x)\] and  denote the set-valued part of $M(p,x)$
as $G(p,x):=\widehat N_{\Gamma(p,x)}(x)$. For $(p,x)$ close to $(\pb,\xb)$ one has
\[G(p,x)=b(p,x)^T N_D(\tilde g(p,x)).\]

The graphical derivative of $G$ is closely related with the graphical derivative of the mapping $\Psi:\R^l\times\R^n\times\R^n\tto\R^n$ given by
\[\Psi(p,x,z):=\widehat N_{\Gamma(p,x)}(z).\]
In order to give a formula for the graphical derivative of $\psi$ we employ the following notation. Given any $y:=(p,x,z)\in \gph\Gamma$ and any $y^*=(p^*,x^*,z^*)\in \widehat N_{\gph \Gamma}(y)$, we denote by
\[\Lambda(y,y^*):=\{\lambda \in N_D(g(y))\mv \nabla g(y)^T\lambda=y^*\}\]
the corresponding set of multipliers and for any $v=(q,u,w)\in\R^l\times\R^n\times\R^n$ by
\[\Lambda(y,y^*;v):=\argmax\{v^T\nabla^2\skalp{\lambda^Tg}(y)v\mv \lambda\in \Lambda(y,y^*)\}\]
the {\em directional} set of multipliers. Further, for any $y^*=(p^*,x^*,z^*)\in\R^l\times\R^n\times\R^n$ we denote by $\pi_3(y^*)$ the canonical projection of $y^\ast$ on its third component, i.e., $\pi_3(y^*)=z^*$.
\begin{proposition}\label{PropGraphDer}
  Assume that Assumption \ref{AssA1} is fulfilled. Then for all $y:=(p,x,z)\in\gph\Gamma$ sufficiently close to $(\pb,\xb,\xb)$, all $z^*\in \Psi(y)$ and all $v:=(q,u,w)\in\R^l\times\R^n\times\R^n$ we have
  \begin{eqnarray*}
    \lefteqn{D\Psi(y,z^*)(v)}\\
    &=&\{\nabla(\nabla_3g(\cdot)^T\lambda)(y)v+\pi_3(N_{\K_{\gph\Gamma}(y,y^*)}(v))\mv y^*\in N_{T_{\gph\Gamma}(y)}(v), \pi_3(y^*)=z^*, \lambda\in\Lambda(y,y^*;v)\}\\
    &=&\{\nabla(\nabla_3g(\cdot)^T\lambda)(y)v+\nabla_3g(y)^T N_{\K_{D}(g(y),\lambda)}(\nabla g(y)v) \mv \lambda\in\Lambda(y,\nabla g(y)^T\mu;v),\\
     &&\hspace{5cm}\nabla_3 g(y)^T\mu =z^*,\ \mu\in N_D\big(g(y)\big),\ \mu^T\nabla g(y)v=0\}.
  \end{eqnarray*}
\end{proposition}
\begin{proof}
 The first equality is an immediate consequence of \cite[Theorem 5.3]{GfrMo17}.
 By $y^*\in N_{T_{\gph\Gamma}(y)}(v)=\widehat N_{\gph\Gamma}(y)\cap [v]^\perp$ we have $y^*=\nabla g(y)^T\mu$ for some $\mu\in N_D\big(g(y)\big)$ with $\mu^T\nabla g(y)v=0$ and due to $\lambda\in \Lambda(y,\nabla g(y)^T\mu;v)$ we also have $\nabla g(y)^T\lambda=y^*$.
 Since
 \[\K_{\gph\Gamma}(y,y^*)=\K_{\gph\Gamma}(y,\nabla g(y)^T\lambda)=\{v\mv \nabla g(y)v\in T_D(g(y)),\ \lambda^T\nabla g(y)v=0\}=\nabla g(y)^{-1}\K_D(g(y),\lambda),\]
 we obtain $\K_{\gph\Gamma}(y,y^*)^\circ=\nabla g(y)^T\K_D(g(y),\lambda)^\circ$ by \cite[Corollary 16.3.2]{Ro70} and by taking into account that the set $\nabla g(y)^T\K_D(g(y),\lambda)^\circ$ is a convex polyhedral cone and therefore closed. Thus
 \begin{align*}N_{\K_{\gph\Gamma}(y,y^*)}(v)&=\K_{\gph\Gamma}(y,y^*)^\circ\cap [v]^\perp
 =\{\nabla g(y)^T\eta \mv \eta \in \K_D(g(y),\lambda)^\circ,\ v^T\nabla g(y)^T\eta=0\}\\
 &=\nabla g(y)^TN_{\K_D(g(y),\lambda)}(\nabla g(y)v)\end{align*}
 showing $\pi_3(N_{\K_{\gph\Gamma}(y,y^*)}(v))=\nabla_3 g(y)^TN_{\K_D(g(y),\lambda)}(\nabla g(y)v)$ and the proof is complete.
 \if{$N_{\K_{\gph\Gamma}(y,\nabla g(y)^T\lambda)}(v)=\Big(\nabla
  g(y)^T\big(N_D(g(y))+[\lambda]\big)\Big)\cap [v]^\perp$, we obtain
\[\pi_3(N_{\K_{\gph\Gamma}(y,y^*)}(v))=\{\nabla_3 g(y)^T \eta \mv \eta\in
N_D(g(y))+[\lambda], \eta^T \nabla g(y)v=0\}.\]
Next, by using the identity
\begin{align*}
\big(N_D(g(y))+[\lambda]\big)\cap [\nabla
g(y)v]^\perp&=\big(T_D(g(y))\cap[\lambda]^\perp\big)^\circ\cap [\nabla g(y)v]^\perp\\
&=\K_{D}(g(y),\lambda)^\circ \cap [\nabla g(y)v]^\perp=N_{\K_{D}(g(y),\lambda)}(\nabla g(y)v),
\end{align*}
we show the second equality and the proof is complete.}\fi
\end{proof}
In what follows we will also use the following multiplier sets
\begin{align*}&\Xi((p,x),x^\ast):=\{\mu\in N_D(\tilde g(p,x))\mv b(p,x)^T\mu=x^*\},\\
&\Xi((p,x),x^\ast;(q,u)):=\{\mu\in \Xi((p,x),x^\ast)\mv \nabla \tilde g(p,x)(q,u)\in \K_D(\tilde g(p,x),\mu)\},\\
&\tilde\Lambda\big((p,x),x^*;(q,u)):=\Big\{\lambda\in\Lambda\big((p,x,x),\nabla g(p,x,x)^T\mu;(q,u,u)\big)\mv \mu\in \Xi((p,x),x^*;(q,u))\Big\}
\end{align*}
defined for $(p,x,x^*)\in \gph G$ and directions $(q,u)\in\R^l\times\R^n$.
\begin{theorem}
  \label{ThGraphDer} Assume that Assumption \ref{AssA1} is fulfilled. Then for all $(p,x)\in\dom G$ sufficiently close to $(\pb,\xb)$, all $x^*\in G(p,x)$
  and all $(q,u)\in\R^l\times\R^n$ we have
  \begin{eqnarray}\label{EqInclGraphDer1}
  D G((p,x),x^*)(q,u)&\subset& D\Psi((p,x,x),x^*)(q,u,u)\\
  \label{EqInclGraphDer2}&=&\Big\{\nabla(b(\cdot)^T\lambda)(p,x)(q,u)+b(p,x)^T N_{\K_{D}(\tilde g(p,x),\lambda)}\big(\nabla \tilde g(p,x)(q,u)\big) \mv \qquad\\
  \nonumber&&\hspace{5cm}\lambda\in\tilde\Lambda\big((p,x),x^*;(q,u)\big)\Big\}.
  \end{eqnarray}
  On the other hand, given $(q,u)\in\R^l\times\R^n$, $\lambda \in \tilde\Lambda\big((p,x),x^*;(q,u))$
  and $\eta \in N_{\K_{D}(\tilde  g(p,x),\lambda)}(\nabla \tilde  g(p,x)(q,u))$,
  assume that the mapping $F:\R^l\times\R^n\times\R^s\tto\R^s\times\R^s$ given by
  \begin{equation} \label{EqF}
  F(p',x',\mu):=\big(\tilde g(p',x'), \mu \big)-\gph N_D
  \end{equation}
  is metrically subregular in direction $(q,u,\eta)$ at $\big((p,x,\lambda),(0,0)\big)$. Then we have
  \begin{equation} \label{EqInclDG1}
     \nabla(b(\cdot)^T\lambda)(p,x)(q,u)+b(p,x)^T\eta\in DG((p,x),x^*)(q,u)
  \end{equation}
  and the tangent $\big(q,u,\nabla(b(\cdot)^T\lambda)(p,x)(q,u)+b(p,x)^T\eta\big)$ to $\gph G$ is derivable.
 \end{theorem}
\begin{proof}
The inclusion \eqref{EqInclGraphDer1} follows immediately from the definition of the graphical derivative, whereas \eqref{EqInclGraphDer2} is a consequence of Proposition \ref{PropGraphDer}.
Consider now $(q,u)\in\R^l\times\R^n$, $\lambda \in \tilde\Lambda\big((p,x),x^*;(q,u))$
and $\eta \in N_{\K_{D}(\tilde g(p,x),\lambda)}(\nabla \tilde g(p,x)(q,u))$ such that the mapping \eqref{EqF}
is directionally metrically subregular. We conclude that
\[(\nabla \tilde g(p,x)(q,u),\eta)\in\gph N_{\K_{D}(\tilde g(p,x),\lambda)}=T_{\gph N_D}(\tilde g(p,x),\lambda)\]
and thus
\[\big(\tilde g(p,x),\lambda\big)+t\big(\nabla \tilde g(p,x)(q,u),\eta\big)\in\gph N_D\]
for all $t\geq 0$ sufficiently small, because $\gph N_D$ is a polyhedral set.

Consequently we have
\[
\dist{(\tilde g(p+tq,x+tu), \lambda + t \eta),\gph N_D} = \oo(t)
\]
and by the assumed directional metric
subregularity of $F$ we can find for every $t>0$  some
$(q_t,u_t,\eta_t)$
with $\lim_{t_\downarrow 0}(q_t,u_t,\eta_t)=(q,u,\eta)$ and $0\in F(p+tq_t,x+tu_t,\lambda + t \eta_t)$ implying
\[
b(p+tq_t,x+tu_t)^T (\lambda + t \eta_t) \in G(p+tq_t,x+tu_t).
\]
On the other hand, by Taylor expansion we obtain
\begin{eqnarray*}
b(p+tq_t,x+tu_t)^T (\lambda + t \eta_t)&=& b(p,x)^T\lambda +t\big(\nabla (b(\cdot)^T\lambda)(p,x)(q,u)+b(p,x)^T\eta\big)+\oo(t)\\
&=&x^*+t\big(\nabla (b(\cdot)^T\lambda)(p,x)(q,u)+b(p,x)^T\eta\big)+\oo(t)
\end{eqnarray*}
showing \eqref{EqInclDG1} and the derivability of the tangent $\big(q,u,\nabla(b(\cdot)^T\lambda)(p,x)(q,u)+b(p,x)^T\eta\big)$ .
\end{proof}
\begin{theorem}\label{ThGraphDerNonDegen}
Assume that Assumption \ref{AssA1} is fulfilled and assume that we are given $(p,x)\in\tilde g^{-1}(D)$ sufficiently close to $(\pb,\xb)$, $x^*\in G(p,x)$ and $(q,u)\in\R^l\times\R^n$  with $\Xi((p,x),x^*;(q,u))\not=\emptyset.$
\begin{enumerate}
  \item[(i)]Assume that for every $\lambda\in \tilde\Lambda((p,x),x^*;(q,u))$ and every $\eta\in N_{\K_D(\tilde g(p,x),\lambda)}(\nabla \tilde g(p,x)(q,u))$ the mapping $F$ given by \eqref{EqF} is metrically subregular in direction $((q,u),\eta)$.
      Then
        \begin{equation}\label{EqDG=DPsi}DG((p,x),x^*)(q,u)=D\Psi((p,x,x),x^*)(q,u,u)\end{equation}
      and all tangents $(q,u,v^*)\in T_{\gph G}((p,x),x^*)$ are derivable.
  \item[(ii)]If the system $\tilde g(\cdot)\in D$ is non-degenerate in direction $(q,u)$ at $(p,x)$  then \eqref{EqDG=DPsi} holds, all tangents $(q,u,v^*)\in T_{\gph G}((p,x),x^*)$ are derivable
   and for all $\mu\in \Xi((p,x),x^*;(q,u))$ the set $\Lambda((p,x,x),\nabla g(p,x,x)^T\mu,(q,u,u))$ is the singleton $\{\mu\}$. Moreover, there is a directional neighborhood ${\cal V}$ of $(q,u)$ such that for all $(p',x')\in ((p,x)+{\cal V})\cap\tilde g^{-1}(D)$, $(p',x')\not=(p,x)$,  the system $\tilde g(\cdot)-D$ is non-degenerate at $(p',x')$ and for every ${x^*}'\in G(p',x')$ we have
  \[DG((p',x'),{x^*}')(q',u')=D\Psi((p',x',x'),{x^*}')(q',u',u')\ \forall (q',u')\in\R^l\times\R^n.\]
\end{enumerate}
\end{theorem}
\begin{proof}(i) follows immediately from Theorem \ref{ThGraphDer}. In order to show the second statement, note  that by Theorem \ref{ThDirNonDegen} the directional non-degeneracy of $\tilde g(\cdot)\in D$ in direction $(q,u)$ implies the assumptions of $(i)$ and therefore \eqref{EqDG=DPsi} follows. In order to show $\Lambda((p,x,x),\nabla g(p,x,x)^T\mu,(q,u,u))=\{\mu\}$ $\forall \mu\in \Xi((p,x),x^*;(q,u))$, fix $\mu\in \Xi((p,x),x^*;(q,u))$ and consider the feasible set
\[T:=\Lambda((p,x,x),\nabla g(p,x,x)^T\mu)=\{\zeta\in N_D(\tilde g(p,x))\mv \nabla g(p,x,x)^T\zeta=\nabla g(p,x,x)^T\mu\}\]
of the linear program defining $\Lambda((p,x,x),\nabla g(p,x,x)^T\mu,(q,u,u))$. We claim that $T=\{\mu\}$. Indeed, $\mu\in T$ and consider any element $\zeta\in T$. Since $\nabla\tilde g(p,x)=(\nabla_1 g(p,x,x),\nabla_2 g(p,x,x)+\nabla_3 g(p,x,x))$, we readily obtain $\nabla\tilde g(p,x)^T\zeta=\nabla\tilde g(p,x)^T\mu$. By definition of $\Xi((p,x),x^*;(q,u))$ we also have $\zeta^T\nabla \tilde g(p,x)(q,u)=\mu^T\nabla \tilde g(p,x)(q,u)=0$ implying  $\zeta,\mu\in N_{T_D(\tilde g(p,x))}(\nabla \tilde g(p,x)(q,u))$. Thus $\nabla\tilde g(p,x)^T(\zeta-\mu)=0$, $\zeta-\mu\in\Span N_{T_D(\tilde g(p,x))}(\nabla \tilde g(p,x)(q,u))$ and we deduce $\zeta-\mu=0$  from the assumed directional non-degeneracy showing $T=\{\mu\}$. Now $\Lambda((p,x,x),\nabla g(p,x,x)^T\mu,(q,u,u))=\{\mu\}$ follows immediately from the definition. The last part of (ii) is implied by Proposition \ref{PropDirNonDegen} taking into account that non-degeneracy of $\tilde g(\cdot)-D$ at $(p',x')$ implies non-degeneracy in any direction and by Remark \ref{RemTrivDir} below.
\end{proof}
\begin{remark}\label{RemTrivDir}Note that in case when $\Xi((p,x),x^*;(q,u))=\emptyset$ we have $D\Psi((p,x,x),x^*)(q,u,u)=\emptyset$ and thus the equality \eqref{EqDG=DPsi} automatically holds by virtue of \eqref{EqInclGraphDer1}. In particular we have $D G((p,x),x^*)(q,u)=D\Psi((p,x,x),x^*)(q,u,u)=\emptyset$ for all directions $(q,u)$ with $\nabla \tilde g(p,x)(q,u)=\nabla g(p,x,x)(q,u,u)\not\in T_D(\tilde g(p,x))$.
\end{remark}
\section{Isolated calmness of the solution mapping}
In what follows 
we define for every $\lambda\in\R^s$ the {\em Lagrangian} $\Lag_\lambda(p,x):\R^l\times\R^n\to \R^n$
by
\[\Lag_\lambda (p,x):=f(p,x) +b(p,x)^T\lambda.\]
\begin{definition}
  We say that the {\em second-order  condition for isolated calmness (SOCIC)} holds at $(\pb,\xb)$ if for every $u\not=0$ and every $\lambda\in \tilde\Lambda\big((\pb,\xb),-f(\pb,\xb);(0,u)\big)$ with
  \[\nabla_2 \tilde g(\pb,\xb)u\in \K_D\big(\tilde g(\pb,\xb)), \lambda)\]
   there exists some $v\in\R^n$ such that
   \[b(\pb,\xb)v\in T_{\K_D(\tilde g(\pb,\xb),\lambda)}\big(\nabla_2\tilde g(\pb,\xb)u\big)\]
   and
  \begin{equation}
    \label{EqSuffSecOrder}v^T\nabla_2\Lag_{\lambda}(\pb,\xb)u<0.
  \end{equation}
\end{definition}

\begin{theorem}Assume that Assumption \ref{AssA1} is fulfilled.
  If SOCIC holds at $(\pb,\xb)$, then the solution map $S$ to the variational system \eqref{EqVarSystem} has the isolated calmness property at $(\pb,\xb)$.

  Conversely, if for every $u\not=0$ there holds
  \begin{equation}\label{EqDQ=DPsi}
        DG\big((\pb,\xb),-f(\pb,\xb)\big)(0,u)=D\Psi\big((\pb,\xb,\xb),-f(\pb,\xb)\big)(0,u,u)
  \end{equation}
     and the mapping $M =f+G$ is metrically subregular in direction $(0,u)$ at $((\pb,\xb),0)$, SOCIC is also necessary for the isolated calmness property of $S$ at $(\pb,\xb)$.
\end{theorem}
\begin{proof}
  We claim that SOCIC is equivalent to the condition
  \begin{equation}\label{EqAuxSuff1}0\in \nabla f(p,x)(0,u)+D\Psi((\pb,\xb,\xb),-f(\pb,\xb))(0,u,u)\ \Rightarrow\ u=0.\end{equation}
   Assume on the contrary that there is some $u\not=0$ such that
   \[0\in \nabla f(p,x)(0,u)+D\Psi((\pb,\xb,\xb),-f(\pb,\xb))(0,u,u).\]
   By \eqref{EqInclGraphDer2}
   this is equivalent to
  \begin{eqnarray}
   \nonumber 0&=& \nabla f(p,x)(0,u)+\nabla \big(b(\cdot)^T\lambda)(\pb,\xb)(0,u)+ b(\pb,\xb)^TN_{\K_D(\tilde g(\pb,\xb),\lambda)}\big(\nabla\tilde g(\pb,\xb)(0,u)\big)\\
  \label{EqAux1}&=&\nabla_2 \Lag_\lambda(\pb,\xb)u+b(\pb,\xb)^TN_{\K_D(\tilde g(\pb,\xb),\lambda)}\big(\nabla_2\tilde g(\pb,\xb)u\big)
  \end{eqnarray}
  for some $\lambda \in\tilde\Lambda((\pb,\xb),-f(\xb); (0,u))$. In particular, $\nabla_2\tilde g(\pb,\xb)u\in \K_D(\tilde g(\pb,\xb),\lambda)$ follows.
  Next observe that
  \begin{eqnarray*}N_{\K_D(\tilde g(\pb,\xb),\lambda)}\big(\nabla_2\tilde g(\pb,\xb)u\big)&=&\K_D(\tilde g(\pb,\xb),\lambda)^\circ\cap[\nabla_2\tilde g(\pb,\xb)u]^\perp=
  \big(\K_D(\tilde g(\pb,\xb),\lambda)+[\nabla_2\tilde g(\pb,\xb)u]\big)^\circ\\
  &=&\Big(T_{\K_D(\tilde g(\pb,\xb),\lambda)}\big(\nabla_2\tilde g(\pb,\xb)u\big)\Big)^\circ
  \end{eqnarray*}
  and thus
  \[b(\pb,\xb)^TN_{\K_D(\tilde g(\pb,\xb),\lambda)}\big(\nabla_2\tilde g(\pb,\xb)u\big)=\big\{v\mv b(\pb,\xb)v\in T_{\K_D(\tilde g(\pb,\xb),\lambda)}\big(\nabla_2\tilde g(\pb,\xb)u\big)\big\}^\circ.\]
   This follows from \cite[Corollary 16.3.2]{Ro70} because the set on the left hand side is a convex polyhedral set and therefore closed. Thus \eqref{EqAux1} is equivalent to
  \[-\nabla_2 \Lag_\lambda(\pb,\xb)u\in \big\{v\mv b(\pb,\xb)v\in T_{\K_D(\tilde g(\pb,\xb),\lambda)}\big(\nabla_2\tilde g(\pb,\xb)u\big)\big\}^\circ\]
  which in turn is equivalent to
  \[-v^T\nabla_2 \Lag_\lambda(\pb,\xb)u\leq 0\ \forall v: b(\pb,\xb)v\in T_{\K_D(\tilde g(\pb,\xb),\lambda)}\big(\nabla_2\tilde g(\pb,\xb)u\big)\]
  contradicting \eqref{EqSuffSecOrder}. Thus the claimed equivalence between SOCIC and  \eqref{EqAuxSuff1} holds true.
  Combining Theorem \ref{ThIsolCalmness} and \eqref{EqInclGraphDer1} we see that the condition \eqref{EqAuxSuff1} and consequently SOCIC as well are sufficient for the isolated calmness property of $S$ at $(\pb,\xb)$.

  In order to show the second statement of the theorem, just note that condition \eqref{EqDQ=DPsi} ensures that \eqref{EqAuxSuff1} and SOCIC are equivalent to the condition
  \[0\in \nabla f(p,x)(0,u)+DG((\pb,\xb,\xb),-f(\pb,\xb))(0,u)\ \Rightarrow\ u=0\]
  and thus by Theorem \ref{ThIsolCalmness} the necessity of SOCIC for the isolated calmness property of $S$ follows.
\end{proof}
By Theorem \ref{ThGraphDerNonDegen}(ii), a sufficient condition for \eqref{EqDQ=DPsi} is that the system $\tilde g(\cdot)\in D$ is non-degenerate in every direction $(0,u)$, $u\not=0$ at $(\pb,\xb)$. We now state a sufficient condition for the metric regularity of the mapping $M=f+G$ in some direction $(q,u)$.
\begin{theorem}\label{ThDirMetrRegM} Let $(q,u)\in\R^l\times\R^n$ and assume that the system $\tilde g(\cdot)\in D$ is non-degenerate in direction $(q,u)$ at $(\pb,\xb)$. Further assume that for every $\hat\lambda\in \Xi((\pb,\xb),-f(\pb,\xb);(q,u))$, every $\eta\in N_{\K_D(\tilde g(\pb,\xb),\hat\lambda)}(\nabla \tilde g(\pb,\xb)(q,u))$ satisfying  $0=\nabla\Lag_{\hat\lambda}(\pb,\xb)(q,u)+b(\pb,\xb)^T\eta$, every pair of faces $\F_1,\F_2$ of the critical cone $\K_D(\tilde g(\pb,\xb),\hat\lambda)$ with $\nabla \tilde g(\pb,\xb)(q,u)\in \F_2\subset \F_1\subset [\eta]^\perp$ and  for  every $0\not=w\in\R^n$ with $b(\pb,\xb)w\in  \F_1-\F_2$ there is some $(\tilde q,\tilde u)$ such that $\nabla\tilde  g(\pb,\xb)(\tilde q,\tilde u)\in \F_1-\F_2$ and
   \[w^T\nabla\Lag_{\hat\lambda}(\pb,\xb)(\tilde q,\tilde u)>0.\]
Then the mapping $M$ is metrically regular in direction $((q,u),0)$ at $((\pb,\xb),0)$.
\end{theorem}
\begin{proof}
  By contraposition. Assume on the contrary that $M=f+G$ is not metrically regular in direction $((q,u),0)$ at $((\pb,\xb),0)$. By virtue of Theorem \ref{ThDirMetrReg} there is some $w\not=0$ such that $(0,0)\in D^*(f+G)\big(((\pb,\xb),0); ((q,u),0)\big)(-w)$. In particular, this implies
  \[0\in DM((\pb,\xb),0)(q,u)=\nabla f(\pb,\xb)(q,u)+DG((\pb,\xb),-f(\pb,\xb))(q,u).\]
  By the definition of the directional limiting coderivative there are sequences $t_k\downarrow 0$, $(q_k,u_k,w_k^*)\to (q,u,0)$ and $(q_k^*,u_k^*,w_k)\to (0,0,w)$ such that
  \[(q_k^*,u_k^*,w_k)\in \widehat N_{\gph(f+G)}((p_k,x_k),t_kw_k^*),\]
   where $p_k:=\pb+t_kq_k$, $x_k:=\xb+t_ku_k$. Hence $((q_k^*,u_k^*)+\nabla f(p_k,x_k)^Tw_k,w_k)\in
\widehat N_{\gph G}((p_k,x_k),x_k^*)$, where $x_k^*:=t_kw_k^*-f(p_k,x_k)$, which is  equivalent to
 \begin{equation}\label{EqPolarity}(q_k^*+\nabla_1f(\pb,\xb)^Tw_k)^T\sigma +(u_k^*+\nabla_2f(\pb,\xb)^Tw_k)^T\xi+w_k^T\xi^*\leq 0\ \forall (\sigma,\xi,\xi^*)\in \gph DG((p_k,x_k),x_k^*).\end{equation}

By Proposition \ref{PropDirNonDegen},  the system $\tilde g(\cdot)\in D$ is non-degenerate at $(p_k,x_k)$ and we deduce from Theorem  \ref{ThGraphDerNonDegen} that $DG((p_k,x_k),x_k^*)(q',u')=D\Psi((p_k,x_k),x_k^*)(q',u',u')$ $\forall (q',u')\in\R^l\times\R^n$. Hence, by taking $\sigma=0$, $\xi=0$ we obtain
\begin{equation} \label{EqAuxPolar}w_k^T b(p_k,x_k)^T\zeta^*\leq 0\ \forall \zeta^*\in \K_D(\tilde g(p_k,x_k),\lambda)^\circ,\ \lambda\in\tilde\Lambda\big((p_k,x_k),x_k^*,(0,0)\big).\end{equation}
Since $\tilde\Lambda\big((p_k,x_k),x_k^*,(0,0)\big)=\{\lambda\in N_D(\tilde g(p_k,x_k))\mv b(p_k,x_k)^T\lambda=x_k^*\}$ and $x_k^*\in G(p_k,x_k)$, by Assumption \ref{AssA1} there exists for every $k$ some $\lambda_k\in \tilde\Lambda\big((p_k,x_k),x_k^*,(0,0)\big)\cap \kappa\norm{x_k^*}\B_{\R^s}$. By passing to a subsequence if necessary we can assume that $\lambda_k$ converges to some $\hat\lambda$. Obviously we have $\hat\lambda\in N_D(\tilde g(\pb,\xb))$ and $b(\pb,\xb)^T\hat\lambda=-f(\pb,\xb)$. By \cite[Lemma 4H.2]{DoRo14}, for each $k$ sufficiently large there are  two closed faces $\F_2^k\subset \F_1^k$ of the critical cone $\K_D(\tilde g(\pb,\xb),\hat\lambda)$ such that
  $\K_D(\tilde g(p_k,x_k),\lambda_k)=\F_1^k-\F_2^k$  and a close look at the proof of \cite[Lemma 4H.2]{DoRo14} tells us that we also have   $\tilde g(p_k,x_k)-\tilde g(\pb,\xb)\in\ri \F_2^k$. Since $\K_D(\tilde g(\pb,\xb),\hat\lambda)$ is a closed convex cone, it has only finitely many faces and by passing to a subsequence once more we can assume $\F_1^k=\F_1$ and $\F_2^k=\F_2$ for all $k$. A face of a closed convex cone is again a cone and thus $(\tilde g(p_k,x_k)-\tilde g(\pb,\xb))/t_k\in\ri \F_2$ $\forall k$. This yields  by passing to the limit that $\nabla \tilde g(\pb,\xb)(q,u)\in \F_2\subset\K_D(\bar g(\pb,\xb),\hat\lambda)$, and consequently $\hat\lambda\in \Xi\big((\pb,\xb),-f(\pb,\xb),(q,u)\big)$.
  \if{ Since
\[\lambda_k\in N_D(\tilde g(p_k,x_k))=N_D(\tilde g(\pb,\xb))\cap[\tilde g(p_k,x_k)-\tilde g(\pb,\xb)]^\perp\subset [\tilde g(p_k,x_k)-\tilde g(\pb,\xb)]^\perp,\]
we obtain
\[\tilde g(p_k,x_k)-\tilde g(\pb,\xb)]\subset \big(T_D(\tilde g(\pb,\xb))+[\tilde g(p_k,x_k)-\tilde g(\pb,\xb)]\big)\cap[\lambda_k]^\perp=K_D(\tilde g(p_k,x_k),\lambda_k)=\F_1-\F_2\subset [\hat \lambda]^\perp\]
implying $\hat\lambda^T\big(\tilde g(p_k,x_k)-\tilde g(\pb,\xb)\big)=0$. Thus for every $\lambda\in N_D(\tilde g(\pb,\xb))$ with $\nabla g(\pb,\xb,\xb)^T(\lambda-\hat\lambda)=0$ we obtain
\begin{eqnarray*}
  0&\geq&\limsup_{k\to\infty}\frac{\lambda^T(\tilde g(p_k,x_k)-\tilde g(\pb,\xb))}{t_k^2}=\limsup_{k\to\infty}\frac{(\lambda-\hat\lambda)^T(\tilde g(p_k,x_k)-\tilde g(\pb,\xb))}{t_k^2}\\
  &=& \limsup_{k\to\infty}\left((\lambda-\hat\lambda)^T\nabla g(\pb,\xb,\xb)\frac{(p_k-\pb,x_k-\xb,x_k-\xb)}{t_k^2}\right.\\
  &&\hspace{5cm}\left.+ \frac 12 (q_k,u_k,u_k)^T\nabla^2\big((\lambda-\hat\lambda)^Tg(\cdot)\big)(\pb,\xb,\xb)(q_k,u_k,u_k)\right)\\
  &=& \frac 12 (q,u,u)^T\nabla^2\big((\lambda-\hat\lambda)^Tg(\cdot)\big)(\pb,\xb,\xb)(q,u,u)
\end{eqnarray*}
showing $\hat\lambda\in \tilde\Lambda\big((\pb,\xb),-f(\pb,\xb),(q,u)\big)$.}\fi
Further we have
\begin{eqnarray*}t_kw_k^*-f(p_k,x_k)&=&x_k^*=b(p_k,x_k)^T\lambda_k=b(\pb,\xb)^T\lambda_k+t_k\nabla (b(\cdot)^T\lambda_k)(\pb,\xb)(q_k,u_k)+\oo(t_k)\\
&=&-f(\pb,\xb)+b(\pb,\xb)^T(\lambda_k-\hat\lambda)+t_k\nabla (b(\cdot)^T\hat\lambda)(\pb,\xb)(q,u)+\oo(t_k),
\end{eqnarray*}
yielding
\begin{eqnarray}\nonumber b(\pb,\xb)^T\frac{\lambda_k-\hat\lambda}{t_k}&=&w_k^*-\frac{f(p_k,x_k)-f(\pb,\xb)}{t_k}-\nabla (b(\cdot)^T\hat\lambda)(\pb,\xb)(q,u)+\oo(t_k)/t_k\\
\label{EqAux2}&=&w_k^*-\nabla\Lag_{\hat\lambda}(q,u)+\oo(t_k)/t_k.\end{eqnarray}
Since $\lambda_k\in N_D(\tilde g(p_k,x_k))\subset N_D(\tilde g(\pb,\xb))$, it holds that $\lambda_k-\hat \lambda$ and consequently $\frac{\lambda_k-\hat\lambda}{t_k}$ belong to $T_{N_D(\tilde g(\pb,\xb))}(\hat \lambda)$. Because of $\nabla \tilde g(\pb,\xb)(q,u)\in\F_2\subset \F_1$ we conclude  $\nabla \tilde g(\pb,\xb)(q,u)\in\F_1-\F_2=\K_D(\tilde g(p_k,x_k),\lambda_k)$ showing $\lambda_k^T\nabla \tilde g(\pb,\xb)(q,u)=0$. Together with $\hat\lambda^T\nabla \tilde g(\pb,\xb)(q,u)=0$ we obtain
\begin{eqnarray*}\frac{\lambda_k-\hat\lambda}{t_k}&\in& T_{N_D(\tilde g(\pb,\xb))}(\hat \lambda)\cap [\nabla \tilde g(\pb,\xb)(q,u)]^\perp=\big(N_D(\tilde g(\pb,\xb))+[\hat \lambda]\big)\cap [\nabla \tilde g(\pb,\xb)(q,u)]^\perp\\
&=&\big(T_D(\tilde g(\pb,\xb))\cap[\hat \lambda]^\perp\big)^\circ\cap [\nabla \tilde g(\pb,\xb)(q,u)]^\perp=N_{\K_D(\tilde g(\pb,\xb),\hat\lambda)}(\nabla \tilde g(\pb,\xb)(q,u)).
\end{eqnarray*}
 Since $\F_1-\F_2= \K_D(\tilde g(p_k,x_k),\lambda_k)\subset [\lambda_k]^\perp$ and $\F_2\subset\F_1\subset \K_D(\tilde g(\pb,\xb),\hat\lambda)$, we have
\[\F_1-\F_2\subset [\lambda_k]^\perp\cap [\hat\lambda]^\perp=\big([\lambda_k]+ [\hat\lambda]\big)^\perp\subset [\lambda_k-\hat\lambda]^\perp\]
and consequently $[\lambda_k-\hat\lambda]\subset (\Span \F_1)^\perp$.
We can now invoke Hoffman's lemma \cite[Theorem 2.200]{BonSh00} to find for every $k$ some $\eta_k\in  N_{\K_D(\tilde g(\pb,\xb),\hat\lambda)}(\nabla \tilde g(\pb,\xb)(q,u))\cap (\Span \F_1)^\perp$ satisfying
\[b(\pb,\xb)^T\frac{\lambda_k-\hat\lambda}{t_k}=b(\pb,\xb)^T\eta_k\]
 and $\norm{\eta_k}\leq \beta\norm{b(\pb,\xb)^T(\lambda_k-\hat\lambda)/{t_k}}$
for some constant $\beta > 0$ not depending on $k$. Since the right hand side of \eqref{EqAux2} is bounded, so is $\eta_k$ and by possibly passing to a subsequence we can assume that $\eta_k$ converges to some $\eta\in  N_{\K_D(\tilde g(\pb,\xb),\hat\lambda)}(\nabla \tilde g(\pb,\xb)(q,u))\cap (\Span \F_1)^\perp$ satisfying
\[b(\pb,\xb)^T\eta=\lim_{k\to\infty}\big(-\nabla \Lag_{\hat\lambda}(\pb,\xb)(q,u)+\oo(t_k)/t_k+w_k^*\big)=-\nabla \Lag_{\hat\lambda}(\pb,\xb)(q,u).\]
From $\eta\in (\Span \F_1)^\perp$ we conclude $\F_1\subset\Span \F_1\subset [\eta]^\perp$.
Moreover, by passing $k$ to infinity in \eqref{EqAuxPolar} it follows that
  \[w^T b(\pb,\xb)^T\zeta^*\leq 0\ \forall \zeta^*\in(\F_1-\F_2)^\circ,\]
   which is the same as $b(\pb,\xb)w\in \F_1-\F_2$. By the assumption of the theorem there is some $(\tilde q,\tilde u)$ with $\nabla\tilde g(\pb,\xb)(\tilde q,\tilde u)\in \F_1-\F_2$ and $w^T\nabla\Lag_{\hat\lambda}(\pb,\xb)(\tilde q,\tilde u)>0$. Applying Corollary \ref{CorUnionFaces} we obtain
   \[\Span N_{T_D(\tilde g(\pb,\xb))}(\nabla \tilde g(\pb,\xb)(q,u))\supset (\F_2-\F_2)^\circ\supset (\F_1-\F_2)^\circ\]
   implying the condition
   \[\nabla\tilde g(\pb,\xb)^T\mu=0,\ \mu\in(\F_1-\F_2)^\circ\ \Rightarrow\ \mu=0.\]
    From Theorem \ref{ThRS} we can deduce that for every $k$ there is some $(\tilde q_k,\tilde u_k)$ satisfying
    \[\nabla\tilde g(p_k,x_k)(\tilde q_k,\tilde u_k)\in \F_1-\F_2=\K_D(\tilde g(p_k,x_k),\lambda_k)\]
    and
    \[\norm{(\tilde q_k,\tilde u_k)-(\tilde q,\tilde u)}\leq \beta'\dist{\nabla\tilde g(p_k,x_k)(\tilde q,\tilde u), \F_1-\F_2}\leq \beta'\norm{(\nabla \tilde g(p_k,w_k)-\nabla \tilde g(\pb,\xb))(\tilde q,\tilde u)}\]
    for some constant $\beta' \geq 0$ not depending on $k$. Since $\tilde g(\cdot)\in D$ is non-degenerate at $(p_k,x_k)$ by Proposition \ref{PropDirNonDegen}, we obtain $\Lambda((p_k,x_k,x_k),x_k^*,(\tilde q_k,\tilde u_k,\tilde u_k))=\{\lambda_k\}$ by Theorem \ref{ThGraphDerNonDegen} and thus
    \[((\tilde q_k,\tilde u_k),\nabla (b(\cdot)^T\lambda_k)(p_k,x_k)(\tilde q_k,\tilde u_k))\in \gph DG((p_x,x_k),x_k^*).\]
     Hence we obtain from \eqref{EqPolarity}
    \begin{align*}&(q_k^*+\nabla_1 f(p_k,x_k)^Tw_k)^T\tilde q_k+(u_k^*+\nabla_2 f(p_k,x_k)^Tw_k)^T\tilde u_k+ w_k^T\nabla (b(\cdot)^T\lambda_k)(p_k,x_k)(\tilde q_k,\tilde u_k)\\
    &= {q_k^*}^T\tilde q_k +{u_k^*}^T\tilde u_k+w_k^T\nabla \Lag_{\lambda_k}(p_k,x_k)(\tilde q_k,\tilde u_k)\leq 0.
    \end{align*}
    By passing $k$ to infinity this yields the contradiction $w^T\nabla \Lag_{\hat\lambda}(\pb,\xb)(\tilde q,\tilde u)\leq 0$ and hence $M$ is metrically regular in direction $((q,u),0)$ at $((\pb,\xb),0)$.
\end{proof}
In case when $(q,u)=(0,0)$ Theorem \ref{ThDirMetrRegM} constitutes a sufficient condition for the metric regularity of $M$ around $((\pb,\xb),0)$. This is an interesting result  for its own sake. On the other hand, when applying Theorem \ref{ThDirMetrRegM} for directions $(0,u)$, $u\not=0$, we have an efficient tool for verifying the necessity of SOCIC for the  isolated calmness property of $S$.

\begin{remark}
  Condition \eqref{EqDQ=DPsi} and the requirement that $M$ is metrically subregular are fulfilled in particular in case of canonical perturbations, i.e., parametric systems given by \eqref{EqVarSystem} with $p=(p_1,p_2)\in\R^n\times\R^s$,
  $f(p,x)=\hat f(x)-p_1$ and $\tilde g(p,x)=\hat g(x)-p_2$.
\end{remark}
\begin{example}
  \label{ExSOCIC}Consider the variational system \eqref{EqVarSystem} with $D:=\R^2_-$ and $f:\R^2\times\R^2\to\R^2$, $g:\R^2\times\R^2\times\R^2\to\R^2$ given by
  \[f(p,x):=\myvec{x_1-p_1\\-x_2},\quad g(p,x,z):=\myvec{p_2-x_1+x_2+z_2\\-x_1-3x_2+z_2}\]
  at $\pb=\xb=(0,0)$. Condition \eqref{EqDualRCQ} ensuring Assumption 1 reads as
  \[\myvec{0\\\mu_1+\mu_2}=\myvec{0\\0},\ \mu_1,\mu_2\geq 0\ \Rightarrow \mu_1=\mu_2=0\]
  and is certainly fulfilled. Further,
  \[b(p,x)=\left(\begin{array}{cc} 0&1\\0&1
  \end{array}\right),\quad \tilde g(p,x)=\myvec{p_2-x_1+2x_2\\-x_1-2x_2}\]
  and for each $p\in\R^2$ the solution set $S(p)$ consists of those $x$ such that there exists some $\lambda\in N_{\R^2_-}(\tilde g(p,x))$ fulfilling
  \[0=\Lag_\lambda(p,x)=\myvec{x_1-p_1\\-x_2+\lambda_1+\lambda_2}.\]
  Straightforward calculations yield that the solution map $S$ is given by
  \begin{equation}\label{EqSolMapEx}S(p)=\begin{cases}\{(p_1,0),\ (p_1,\frac{p_1-p_2}2)\}&\mbox{if $p_2-p_1\leq 0$, $p_1\geq0$},\\
  \{(p_1,-\frac{p_1}2),\ (p_1,\frac{p_1-p_2}2)\}&\mbox{if $p_2-2p_1\leq 0$, $p_1< 0$},\\
  \emptyset&\mbox{otherwise.}\end{cases}\end{equation}
  We see that $S$ has the isolated calmness property at $(\pb,\xb)$ and we now want to verify that SOCIC is fulfilled. Consider $u\not=0$ such that
  \[\nabla_2 \tilde g(\pb,\xb)u=\myvec{-u_1+2u_2\\-u_1-2u_2}\in T_D(\tilde g(\pb,\xb))=\R^2_-.\]
  In particular we have $u_1\not=0$ because $u_1=0$ implies $u_2=0$ and the case $u=0$ is excluded.
  Since $\Xi\big((\pb,\xb),-f(\pb,\xb)\big)=\{\mu\in\R^{2}_+\mv (0,\mu_1+\mu_2)=(0,0)\}=\{(0,0)\}$, we have $\Xi\big((\pb,\xb),-f(\pb,\xb);(0,u)\big)=\tilde\Lambda\big((\pb,\xb),-f(\pb,\xb);(0,u)\big)=\{(0,0)\}$.
  By choosing $v=(-u_1,0)$ we have $b(\pb,\xb)v=0\in T_{\K_D(\tilde g(\pb,\xb),\lambda)}(\nabla_2\tilde g(\pb,\xb)u)$ and $v^T\nabla_2\Lag_0(\pb,\xb)u=-u_1^2<0$ and SOCIC is established.

  Next we show that the mapping $(p,x)\rightrightarrows M(p,x)=f(p,x)+G(p,x)$ is metrically regular around  $((\pb,\xb),0)$ by applying Theorem \ref{ThDirMetrRegM} with $(q,u)=(0,0)$. The Jacobian $\nabla \tilde g(\pb,\xb)$ has full row rank and hence the system $\tilde g(\cdot)\in \R^2_-$ is non-degenerate. It can be easily deduced that the only $\hat\lambda\in \Xi((\pb,\xb),-f(\pb,\xb);(0,0))$ and the only $\eta\in N_{\K_D(\tilde g(\pb,\xb),\hat\lambda)}(\nabla \tilde g(\pb,\xb)(0,0))$ satisfying  $0=\nabla\Lag_{\hat\lambda}(\pb,\xb)(0,0)+b(\pb,\xb)^T\eta$ are  $\hat\lambda=\eta=(0,0)$. We have to show that for every  pair of faces $\F_2\subset\F_1\subset\R^2_-$ and every $0\not=w\in\R^2$ satisfying $b(\pb,\xb)w=(w_2,w_2)^T\in\F_1-\F_2$ there is some $(\tilde q,\tilde u)$ with
  \begin{equation}\label{EqEx1}\nabla \tilde g(\pb,\xb)(\tilde q,\tilde u)=\myvec{\tilde q_2-\tilde u_1+2\tilde u_2\\-\tilde u_1-2\tilde u_2}\in \F_1-\F_2\ \mbox{and}\ w^T\nabla\Lag_0(\pb,\xb)(\tilde q,\tilde u)=w_1(\tilde u_1-\tilde q_1)-w_2\tilde u_2>0.\end{equation}
  If $w_1\not=0$, this can be easily accomplished by taking $\tilde q_2=\tilde u_1=\tilde u_2=0$ and $\tilde q_1=-w_1$. So let $w_1=0$ and consequently $w_2\not=0$. Then condition \eqref{EqEx1} is fulfilled when we take, e.g., $\tilde u_2=-w_2$, $\tilde u_1=-2\tilde u_2$, $\tilde q_2=\tilde u_1-2\tilde u_2$ and an arbitrary $\tilde q_1$. So, we have detected the metric regularity of $M$ from Theorem \ref{ThDirMetrRegM}.\hfill$\triangle$
\end{example}
\section{On the Aubin property of the solution map}
In the following theorem we state our main result concerning the Aubin property of the solution map $S$ relative to some set $P$.

\begin{theorem}\label{ThFinalRelAubin}Assume that Assumption \ref{AssA1} is fulfilled and we are given a closed set $P\subset \R^l$ containing $\pb$ such that the following assumptions are fulfilled:
  \begin{enumerate}
  \item[(i)]For every $q\in T_P(\pb)$ there is some  $u\in\R^n$ such that
  \begin{eqnarray}
    \label{EqSurjGraphDer2} 0\in \nabla f(\pb,\xb)(q,u)+D\Psi\big((\pb,\xb,\xb),-f(\pb,\xb)\big)(q,u,u).
  \end{eqnarray}
  \item[(ii)]For every $(0,0)\not=(q,u)$ verifying \eqref{EqSurjGraphDer2} the (partial) directional non-degeneracy condition
  \begin{equation}
    \label{EqPartNonDegen}\nabla_2\tilde g(\pb,\xb)^T\mu=0,\ \mu\in \Span N_{T_D(\tilde g(\pb,\xb))}(\nabla \tilde g(\pb,\xb)(q,u))\ \Rightarrow\ \mu=0
  \end{equation}
  is fulfilled and for every $\hat\lambda\in \Xi((\pb,\xb),-f(\pb,\xb);(q,u))$,  every $\eta\in N_{\K_D(\tilde g(\pb,\xb),\hat\lambda)}(\nabla \tilde g(\pb,\xb)(q,u))$ satisfying  $0=\nabla\Lag_{\hat\lambda}(\pb,\xb)(q,u)+b(\pb,\xb)^T\eta$, every pair of faces $\F_1,\F_2$ of the critical cone $\K_D(\tilde g(\pb,\xb),\hat\lambda)$ with $\nabla \tilde g(\pb,\xb)(q,u)\in \F_2\subset \F_1\subset[\eta]^\perp$ and  every $w\not=0$ with $b(\pb,\xb)w\in  \F_1-\F_2$ there is some $\tilde w$ with $\nabla_2\tilde  g(\pb,\xb)\tilde w\in \F_1-\F_2$ such that
   \[w^T\nabla _2\Lag_{\hat\lambda}(\pb,\xb)\tilde w>0.\]
  \end{enumerate}
     Then the solution mapping $S$ to the variational system \eqref{EqVarSystem} has the Aubin property relative to $P$ around $(\pb,\xb)$ and for every $q\in T_P(\pb)$ there holds
   \[DS(\pb,\xb)(q)=\{u\mv 0\in \nabla f(\pb,\xb)(q,u)+D\Psi\big((\pb,\xb,\xb),-f(\pb,\xb)\big)(q,u,u)\}.\]
\end{theorem}

\begin{proof}We will invoke Theorem \ref{ThRelAubin} in order to prove the proposition. Observe that \eqref{EqPartNonDegen} implies the non-degeneracy of the system $\tilde g(\cdot)\in D$ in direction $(q,u)$ at $(\pb,\xb)$ and by Theorem \ref{ThGraphDerNonDegen} we have that $D\Psi\big((\pb,\xb,\xb),-f(\pb,\xb)\big)(q,u,u)=DG\big((\pb,\xb),-f(\pb,\xb)\big)(q,u)$ and all tangents $(q,u,u^*)$ to $\gph G$ at $((\pb,\xb),-f(\pb,\xb))$ are derivable.
Since $DM((\pb,\xb),0)(q,u)=\nabla f(\pb,\xb)(q,u)+DG((\pb,\xb),-f(\pb,\xb))(q,u)$ and taking into account Remark \ref{RemDerivTangents}, assumption (i) of Theorem \ref{ThRelAubin} is satisfied due to the first assumption.

We now show that assumption (ii) of Theorem \ref{ThRelAubin} is fulfilled as well. Assume that we are given a direction $(0,0)\not=(q,u)$ satisfying
\begin{eqnarray*}0&\in&DM((\pb,\xb),0)(q,u)=\nabla f(\pb,\xb)(q,u)+DG((\pb,\xb),-f(\pb,\xb))(q,u)\\
&\subset& \nabla f(\pb,\xb)(q,u)+D\Psi((\pb,\xb,\xb),-f(\pb,\xb))(q,u,u) \end{eqnarray*}
and $(q^*,w)$ such that
$(q^*,0)\in D^*M\big(((\pb,\xb),0);((q,u),0)\big)(-w).$

 By the definition of the directional limiting coderivative there are sequences $t_k\downarrow 0$, $(q_k,u_k,w_k^*)\to (q,u,0)$ and $(q_k^*,u_k^*,w_k)\to (q^*,0,w)$ such that
  \[(q_k^*,u_k^*,w_k)\in \widehat N_{\gph(f+G)}((p_k,x_k),t_kw_k^*),\]
   where $p_k:=\pb+t_kq_k$, $x_k:=\xb+t_ku_k$. We can now proceed as in the proof of Theorem \ref{ThDirMetrRegM} to find the sequences $x_k^*$ and  $\lambda_k$ as well as \[\hat\lambda=\lim_{k\to\infty}\lambda_k\in\Xi((\pb,\xb),-f(\pb,\xb);(q,u)),\ \eta\in N_{\K_D(\tilde g(\pb,\xb),\hat\lambda)}(\nabla \tilde g(\pb,\xb)(q,u))\] with $0=\nabla\Lag_{\hat\lambda}(\pb,\xb)(q,u)+b(\pb,\xb)^T\eta$ and the faces $\F_1,\F_2$ of the critical cone $\K_D(\tilde g(\pb,\xb),\hat\lambda)$ with $\nabla \tilde g(\pb,\xb)(q,u)\subset \F_2\subset \F_1\subset [\eta]^\perp$ such that $\K_D(\tilde g(p_k,x_k),\lambda_k))=\F_1-\F_2$ $\forall k$.
  As in the proof of Theorem \ref{ThDirMetrRegM} we can also deduce $b(\pb,\xb)w\in \F_1-\F_2$. By assumption (ii) of the theorem there is some $\tilde w$ with $\nabla_2\tilde g(\pb,\xb)\tilde w\in \F_1-\F_2$ and $w^T\nabla_2\Lag_{\hat\lambda}(\pb,\xb)\tilde w>0$, provided $w\not=0$, which we now assume.

  Next observe that the implication
  \begin{equation}\label{EqMetrRegSecComp}\nabla_2g(\pb,\xb)^T\mu=0,\ \mu\in (\F_1-\F_2)^\circ\ \Rightarrow\ \mu=0\end{equation}
  follows from \eqref{EqPartNonDegen} by virtue of Corollary \ref{CorUnionFaces}.
  By condition \eqref{EqMetrRegSecComp} and Theorem \ref{ThRS}, there is some real $\beta>0$ such that for every $k$ sufficiently large
  there are some $\tilde w_k$ satisfying
  $\nabla_2\tilde g(p_k,x_k)\tilde w_k\in\F_1-\F_2$ and
  \[\norm{\tilde w_k-\tilde w}\leq \beta\dist{\nabla_2\tilde g(p_k,x_k)\tilde w,\F_1-\F_2}\leq \beta \norm{(\nabla_2\tilde g(p_k,x_k)-\nabla_2\tilde g(\pb,\xb))\tilde w}.\]
  Hence $\nabla \tilde g(p_k,x_k)(0,\tilde w_k)\in \F_1-\F_2=\K_D(\tilde g(p_k,x_k),\lambda_k)$ and, since $\tilde g(\cdot)\in D$ is non-degenerate at $(p_k,x_k)$, we obtain $\Lambda((p_k,x_k,x_k),x_k^*;(0,\tilde w_k,\tilde w_k))=\{\lambda_k\}$ by Theorem \ref{ThGraphDerNonDegen}. Using Theorem \ref{ThGraphDerNonDegen} once more together with \eqref{EqInclGraphDer2} we obtain
   \[(0,\tilde w_k,\nabla (b(\cdot)^T\lambda_k)(p_k,x_k)(0,\tilde w_k))\in \gph DG((p_k,x_k),x_k^*),\]
   yielding
   \[(u_k^*+\nabla_2 f(p_k,x_k)^Tw_k)^T\tilde w_k+ w_k^T\nabla_2 (b(\cdot)^T\lambda_k)(p_k,x_k)\tilde w_k={u_k^*}^T\tilde w_k+w_k^T\nabla_2 \Lag_{\lambda_k}(p_k,x_k)\tilde w_k\leq 0\]
   from \eqref{EqPolarity}. By passing to the limit we obtain the contradiction $w^T\nabla_2 \Lag_{\hat\lambda}\tilde w\leq 0$ and thus $w=0$.
   It remains to show that $q^*=0$. Observe that \eqref{EqMetrRegSecComp} is equivalent to
   \[(\ker \nabla_2 g(\pb,\xb)^T\cap (\F_1-\F_2)^\circ)^\circ= \nabla_2g(\pb,\xb)\R^n+(\F_1-\F_2)=\{0\}^\circ=\R^s.\]
   Hence there is some $\bar u\in\R^n$ with $\nabla_1\tilde g(\pb,\xb)q^*+\nabla_2\tilde g(\pb,\xb)\bar u\in\F_1-\F_2$. Further, by assumption \eqref{EqMetrRegSecComp} and Theorem \ref{ThRS}, there is some real $\beta>0$ such that for every $k$ sufficiently large there exist some vectors $\bar u_k$ satisfying $\nabla_1\tilde g(p_k,x_k)q^*+\nabla_2\tilde g(p_k,x_k)\bar u_k\in\F_1-\F_2$ and $\norm{\bar u_k-\bar u}\leq \beta\dist{\nabla_1\tilde g(p_k,x_k)q^*+\nabla_2\tilde g(p_k,x_k)\bar u,\F_1-\F_2}\leq \beta \norm{(\nabla\tilde g(p_k,x_k)-\nabla\tilde g(\pb,\xb))(q^*,\bar u)}$. Using similar arguments as before we deduce
   \[(q^*,\bar u_k, \nabla (b(\cdot)^T\lambda_k)(p_k,x_k)(q^*,\bar u_k))\in \gph DG((p_k,x_k),x_k^*),\]
   resulting in
   \begin{align*}&(q_k^*+\nabla_1 f(p_k,x_k)^Tw_k)^Tq^*+ (u_k^*+\nabla_2 f(p_k,x_k)^Tw_k)^T\bar u_k+w_k^T\nabla(b(\cdot)^T\lambda_k)(p_k,x_k)(q^*,\bar u_k)\\
   &={q_k^*}^Tq^*+{u_k^*}^T\bar u_k +w_k^T\nabla \Lag_{\lambda_k}(p_k,x_k)(q^*,\bar u_k)\leq 0\end{align*}
   by means of \eqref{EqPolarity}. By passing $k$ to infinity we obtain ${q^*}^Tq^*\leq 0$ implying $q^*=0$. Thus all assumptions of Theorem \ref{ThRelAubin} are fulfilled and the statement is established.
\end{proof}

In case when $P=\R^s$ Theorem \ref{ThFinalRelAubin} improves \cite[Theorem 6]{GfrOut18} by weakening the assumption that the multiplier $\lambda\in N_D(\tilde g(p,x))$ satisfying $b(\pb,\xb)^T\lambda=-f(\pb,\xb)$ is unique.

\begin{example}\label{ExRelAubin}It is easy to see that for the variational system of Example \ref{ExSOCIC} the solution map $S$ given by \eqref{EqSolMapEx} has the Aubin property relative to its domain $\dom S=\{(p_1,p_2)\mv p_2-p_1\leq 0, p_2-2p_1\leq 0\}$. We want now to analyze the conditions on the set $P$ provided by Theorem \ref{ThFinalRelAubin} such that $S$ has the Aubin property relative to $P$. After some calculations we obtain the following table where we list all directions $(q,u)$ such that \eqref{EqSurjGraphDer2} holds as well as $\hat\lambda\in\Xi\big((\pb,\xb), -f(\pb,\xb);(q,u)\big)$ and $\eta\in N_{\K_D(\tilde g(\pb,\xb),\hat\lambda)}(\nabla \tilde g(\pb,\xb)(q,u))$ such that $0=\nabla\Lag_{\hat\lambda}(\pb,\xb)(q,u)+b(\pb,\xb)^T\eta$. In addition we display vector $\nabla \tilde g(\pb,\xb)(q,u)$, useful also for
the computation of $\eta$.
\begin{equation}\label{EqTable}\begin{array}{|c|c|c|c|c|}\hline
q&u&\hat\lambda&\big(\nabla \tilde g(\pb,\xb)(q,u)\big)^T&\eta\in\R^2_+\\\hline
 q_1\geq 0,q_2-q_1\leq 0&(q_1,0)&0&(q_2-q_1,-q_1)&(0,0)\\
q_2-q_1\leq0, q_2-2q_1<0&(q_1,\frac{q_1-q_2}2)&0&(0,q_2-2q_1)&(\frac{q_1-q_2}2,0)\\
q_1\leq 0, q_2=2q_1&(q_1,\frac{-q_1}2)&0&(0,0)&\eta_1+\eta_2=\frac{-q_1}2\\
q_1\leq 0, q_2-2q_1 <0& (q_1,\frac{-q_1}2)&0&(q_2-2q_1,0)&(\frac{-q_1}2,0)\\\hline
\end{array}\end{equation}
From this table we see that condition (i) of Theorem \ref{ThFinalRelAubin} amounts to the requirement that
\[T_P(\pb)\subset \{q\in\R^2\mv q_2-q_1\leq0, q_2-2q_1\leq 0\}(=\dom S).\]

In the next step we will analyze condition (ii) of Theorem \ref{ThFinalRelAubin}. Since $\nabla_2\tilde g(\pb,\xb)$ has full rank, implication \eqref{EqPartNonDegen} holds for any direction $(q,u)$.
Consider now $(0,0)\not=(q,u)$ together with $\hat\lambda=0$ and $\eta$ from table \eqref{EqTable} and faces $\F_1,\F_2$ of the critical cone $\K_D(\tilde g(\pb,\xb),\hat\lambda)=\R^2_-$ satisfying $\nabla \tilde g(\pb,\xb)(q,u)\subset\F_2\subset\F_1\subset[\eta]^\perp$. Observe that $(q,u)\not=(0,0)$ implies $q\not=0$. Further, consider $0\not=w\in\R^2$ such that $b(\pb,\xb)w=(w_2,w_2)^T\in \F_1-\F_2$. It follows that $w_2=0$ whenever $\eta\not=0$. Further, $\F_1=\R^2_-$ when $w_2\not=0$ and $\F_1-\F_2=\R^2$ if $w_2>0$. Our next analysis is split into three cases.

{\bf Case $\mathbf{w_2>0}$:} Then $\F_1-\F_2=\R^2$ and obviously $\tilde w=(0,-w_2)$ fulfills $\nabla_2 \tilde g(\pb,\xb)\tilde w\in\R^2$ and
\[w^T\nabla_2\Lag_{\hat\lambda}(\pb,\xb)\tilde w=w_2^2>0.\]

{\bf Case $\mathbf{w_2<0}$:} It follows that $\eta=0$ and $\F_1=\R^2_-$. If $\nabla \tilde g_1(\pb,\xb)(q,u)<0$ then $\F_1-\F_2\supset\R\times \R_-$ and we can take $\tilde w=(0,-w_2)$ to obtain $\nabla_2 \tilde g(\pb,\xb)\tilde w=(-2w_2,2w_2)^T\in\R\times \R_-\subset\F_1-\F_2$ and $w^T\nabla_2\Lag_{\hat\lambda}(\pb,\xb)\tilde w=w_2^2>0.$ Hence assume that $\nabla \tilde g_1(\pb,\xb)(q,u)=0$. A look at table \eqref{EqTable} tells us that this together with $\eta=0$ and $q\not=0$ is only possible when $q_2=q_1$ and $q_1>0$.  In this case we can take $w=(-1,-2)$ and $\F_2=\{0\}\times\R_-$, $\F_1=\R^2_-$ resulting in $\F_1-\F_2=\R_-\times\R$ and it follows that there does not exist any
\[\tilde w\in\{\tilde w\mv \nabla_2\tilde g(\pb,\xb)\tilde w\in\F_1-\F_2\}=\{\tilde w\mv -\tilde w_1+2\tilde w_2\leq 0\}\]
fulfilling
\[w^T\nabla_2\Lag_{\hat\lambda}(\pb,\xb)\tilde w=-\tilde w_1+2\tilde w_2>0.\]

{\bf Case $\mathbf{w_2=0}$:} Note that $w\not=0$ implies $w_1\not=0$. If $\nabla \tilde g_1(\pb,\xb)(q,u)<0$ then $\F_1-\F_2\supset\R\times \{0\}$ and we can take $\tilde w=(w_1,w_1/2)$ to obtain
$\nabla_2\tilde g(\pb,\xb)\tilde w\in\{0\}\times\R\subset \F_1-\F_2$ and $w^T\nabla_2\Lag_{\hat\lambda}(\pb,\xb)\tilde w=w_1^2>0$. If $\nabla \tilde g_2(\pb,\xb)(q,u)<0$, then we can argue as before to show that $\tilde w=(w_1,-w_1/2)$ fulfills the required conditions. There remains the case that $\nabla \tilde g(\pb,\xb)(q,u)=(0,0)$. A look at Table \ref{EqTable} shows that this is possible for nonzero $q$ only in case when $q_2=2q_1$ and $q_1<0$. Taking $\eta=(-q_1/2,0)$, $\F_2=\F_1=\{(0,0)\}$, we obtain that the only $\tilde w$ with $\nabla_2\tilde g(\pb,\xb)\tilde w\in\F_1-\F_2$ is $\tilde w=(0,0)$ and therefore we again cannot fulfill the condition $w^T\nabla_2\Lag_{\hat\lambda}(\pb,\xb)\tilde w>0$.

The above analysis  shows that we have to exclude the sets $P$ such that
\[T_P((0,0))\cap  \{(t,t),\ (-t,-2t)\mv t>0\}\not=\emptyset.\]
This means that, by virtue of Theorem \ref{ThFinalRelAubin}, $S$ has the Aubin property relative to $P$ around $(\pb,\xb)$ for every closed set $P$ containing $(0,0)$ such that
\[ T_P((0,0))\subset \{q\in\R^2\mv q_2-q_1\leq 0,\ q_2-2q_1\leq 0\}\setminus \{(t,t),\ (-t,-2t)\mv t>0\}(=\inn\dom S\cup\{(0,0)\}).\]
\hfill$\triangle$
\end{example}

\section{Conclusion}

In most rules of generalized differentiation one associates with the data a certain mapping and requires, as a qualification condition, the metric subregularity of this mapping at the considered point, see, e.g., \cite{Iof79d,Iof79a,HenJouOut02,IofOut08}. Correspondingly, in the directional limiting calculus the qualification conditions amount typically to the directional metric subregularity of the respective associated mappings, cf. \cite{BeGfrOut18}. In both cases, however, the required property should be verifiable via suitable conditions expressed solely in terms of problem data. In this paper we construct such conditions on the basis of the (stronger) property of directional metric regularity, see Theorems \ref{ThDirNonDegen}, \ref{ThRelAubin}, \ref{ThDirMetrRegM} and \ref{ThFinalRelAubin}.

In general, the principal questions related to metric subregularity, calmness and the associated areas of error bounds and subtransversality have been thoroughly investigated by many notable researchers including A.~Y.~ Kruger (\cite{DoGfrKruOut18, FabHenKruOut10, FabHenKruOut12, Kru15a} and many other works on this subject). Via the research, discussed in this paper, the authors would like to give credit to their friend Alex on the occasion of his 65$^{\rm th}$ birthday.

\section*{Acknowledgements} The research of the first two authors was  supported by the Austrian Science Fund (FWF) under grant P29190-N32. The  research of the third author was supported by the Grant Agency of the Czech Republic, Projects 17-08182S and 17-04301S and the Australian Research Council, Project DP160100854.

\end{document}